\documentclass[12pt,dvips]{amsart}
\usepackage{amssymb}
\usepackage[all,line]{xy}
\usepackage{amsfonts}
\usepackage{tikz}
\usetikzlibrary{arrows}

\newtheorem{thm}{Theorem}[section]
\newtheorem{conj}{Conjecture}[section]
\newtheorem{cor}{Corollary}[section]
\newtheorem{lem}{Lemma}[section]
\newtheorem{obs}{Observation}[section]
\newtheorem{prp}{Proposition}[section]
\theoremstyle{definition}
\newtheorem{ques}{Question}[section]
\newtheorem{rmk}{Remark}[section]
\newtheorem{defn}{Definition}[section]

\DeclareMathOperator{\area}{area}
\DeclareMathOperator{\rk}{rank}
\DeclareMathOperator{\M}{Motz}
\DeclareMathOperator{\Dy}{Dyck}

\DeclareMathOperator{\B}{\mathcal{B}}
\DeclareMathOperator{\inv}{inv}

\DeclareMathOperator{\exc}{exc}
\DeclareMathOperator{\nzb}{\text{nzb}}
\DeclareMathOperator{\maj}{maj}

\title[Area and rank of a Dyck path]{Counting Dyck paths by area and rank}

\author[S. A. Blanco and T. K. Petersen]{Sa\'ul A. Blanco and T. Kyle Petersen}

\address{Department of Mathematical Sciences, DePaul University, Chicago IL 60614}

\email{sblancor@depaul.edu, tkpeters@math.depaul.edu}

\begin{document}

\maketitle

\begin{abstract}
The set of Dyck paths of length $2n$ inherits a lattice structure from a bijection with the set of noncrossing partitions with the usual partial order. In this paper, we study the joint distribution of two statistics for Dyck paths: \emph{area} (the area under the path) and \emph{rank} (the rank in the lattice). 

While area for Dyck paths has been studied, pairing it with this rank function seems new, and we get an interesting $(q,t)$-refinement of the Catalan numbers. We present two decompositions of the corresponding generating function: one refines an identity of Carlitz and Riordan; the other refines the notion of $\gamma$-nonnegativity, and is based on a decomposition of the lattice of noncrossing partitions due to Simion and Ullman.

Further, Biane's correspondence and a result of Stump allow us to conclude that the joint distribution of area and rank for Dyck paths equals the joint distribution of length and reflection length for the permutations lying below the $n$-cycle $(12\cdots n)$ in the absolute order on the symmetric group.
\end{abstract}

\section{Introduction}

The \emph{Catalan numbers} $C_n = \frac{1}{n+1}\binom{2n}{n}$ count many sets of combinatorial objects; see \cite[Exercise 6.19]{EC2}. Chief among these are \emph{Dyck paths} and \emph{noncrossing partitions}.

The area statistic for Dyck paths is well studied. For example, we have the following refined Catalan identity due to Carlitz and Riordan \cite{CR64}:
\begin{equation}\label{eq:area} \Dy(n;q) := \sum_{p \in \Dy(n)} q^{\area(p)} = \sum_{k=0} q^k \Dy(k;q)\Dy(n-1-k;q),
\end{equation} 
where $\Dy(n)$ denotes the set of Dyck paths of length $2n$.

Similarly, the rank function for the lattice of noncrossing partitions of $\{1,2,\ldots,n\}$ is well understood. For instance, rank is $n$ minus the number of blocks, and the Narayana numbers $N_{n,k}=\frac{1}{k}\binom{n-1}{k-1}\binom{n}{k-1}$ count the number of noncrossing partitions of $n$ with $k$ blocks. One of the more interesting results about this function is the following identity, which can be found in \cite[Proposition 11.14]{PRW08}. Let $NC(n)$ denote the set of noncrossing partitions. Then,
\begin{equation}\label{eq:rank}
NC(n;t) := \sum_{\pi \in NC(n)} t^{\rk(\pi)} = \sum_{0\leq j\leq (n-1)/2} \gamma_j t^j (1+t)^{n-1-2j},
\end{equation}
where the $\gamma_j$ are nonnegative integers. (One combinatorial interpretation shows $\gamma_j$ is the number of $132$-avoiding permutations with $j$ descents and $j$ peaks.) Notice that both symmetry and unimodality of the Narayana numbers (i.e., $N_{n,k} = N_{n,n+1-k}$ and $N_{n,1} \leq \cdots \leq N_{n,\lceil n/2 \rceil} \geq \cdots \geq N_{n,n}$) follow from this result. Further, it can be shown that $\gamma_j = C_j\binom{n-1}{2j}$, so we get the following Catalan identity by setting $t=1$ in Equation \eqref{eq:rank} (see \cite[Corollary 3.1]{SU91}): \[ C_n = \sum_{0\leq j \leq (n-1)/2} C_j\binom{n-1}{2j}2^{n-1-2j}.\] See \cite{NP11, PetersenShards, PRW08} for extensive discussion of the numbers $\gamma_j$ in the expansion, and of generating functions exhibiting similar expansions.

\begin{rmk}
A polynomial that is in the nonnegative span of the polynomials $\Gamma_d = \{ t^j(1+t)^{d-2j} : 0\leq j \leq d/2\}$ (for fixed $d$) is called \emph{$\gamma$-nonnegative}. This property has gained interest since work of Gal \cite{G05} showed that for $h$-polynomials of simplicial spheres, $\gamma$-nonnegativity implies the Charney-Davis conjecture. The polynomial $NC(n;t)$ above is the $h$-polynomial of the \emph{associahedron}, so \eqref{eq:rank} is an example satisfying Gal's condition. See also \cite{B04/06, NP11, NPT11, PRW08, St08}.
\end{rmk}

In this paper, we study a refinement of both $\Dy(n;q)$ and $NC(n;t)$, that yields, among other things, refined versions of Equation \eqref{eq:area} and Equation \eqref{eq:rank}.

We interpret this bivariate polynomial in two ways. Primarily, we consider the joint distribution of area and rank for Dyck paths (where the rank is defined via a bijection with noncrossing partitions), and define \[ \Dy(n;q,t) = \sum_{ p \in \Dy(n)} q^{\area(p)}t^{\rk(p)}.\] As shown in Corollary \ref{cor:ll'}, we can also interpret this polynomial as the joint distribution of length and reflection length for permutations in the interval $[e,(12\cdots n)]$ in the absolute order on $S_n$.

\begin{rmk}
While setting $q=t=1$ in $\Dy(n;q,t)$ yields the Catalan number $C_n$, this polynomial should not be confused with the ``$(q,t)$-Catalan polynomials" that arise in the study of MacDonald polynomials. See Haglund's manuscript \cite{H08} for a thorough treatment. 
\end{rmk}

\begin{rmk}
Before stating our main results, we wish to point out that while there is a natural bijection that takes a noncrossing partition with $k$ blocks to a Dyck path with $k$ peaks, the bijection we use, first due to Stump \cite{Stump}, is different. In particular, the path with one peak corresponds to the noncrossing partition $\{\{1,n\}, \{2,n-1\},\ldots\}$ of rank $\lfloor n/2\rfloor$.
\end{rmk}

Our first result is to show that these polynomials satisfy a Catalan-style recurrence.

\begin{thm}[A refined Catalan recurrence]\label{thm:recurrence}
For $n\geq 1$, we have \[ \Dy(n;q,t) = \sum_{k=0}^{n-1} (qt)^k \Dy(k; q, 1/t)\Dy(n-1-k; q, t).\]
\end{thm}

We use Theorem \ref{thm:recurrence} to derive a continued fraction expansion for the ordinary generating function \[ \Dy(q,t,z) = \sum_{n\geq 0} \Dy(n;q,t) z^n\] in Proposition \ref{prp:cf}.

One way to prove the $\gamma$-nonnegativity of Equation \ref{eq:rank} is to use the idea of a \emph{symmetric boolean decomposition} of the lattice of noncrossing partitions, found in work of Simion and Ullman \cite{SU91}. See \cite{PetersenShards,PRW08}. Loosely, this is a partition of a poset $P$ into disjoint boolean intervals whose middle ranks coincide with the middle rank of $P$. In Section \ref{sec:SU} we revisit Simion and Ullman's decomposition to show the following decomposition of $\Dy(n;q,t)$.

\begin{thm}[Refined $\gamma$-nonnegativity]\label{thm:main}
For $n\geq 1$, we have \[ \Dy(n;q,t) = \sum_{m \in \M(n)} q^{\area(m)}t^{\rk(m)}(1+qt)^{n-1-2\rk(m)},
  \]
where $\M(n)$ denotes the set of Motzkin paths of order $n$ (length $n-1$).
\end{thm} 

It is easily shown that $\area(p) \geq \rk(p)$, so from Theorem \ref{thm:main} we can write
\begin{equation}\label{eq:qgam}
 \Dy(n;q,t/q) = \sum_{p \in \Dy(n)} q^{\area(p)-\rk(p)}t^{\rk(p)} = \sum_{0\leq j \leq (n-1)/2} \gamma_j(q) t^j(1+t)^{n-1-2j},
\end{equation}
where the $\gamma_j(q)$ count Motzkin paths of rank $j$ according to area. In other words, we have written our polynomial in the ``gamma basis" $\Gamma_{n-1}$ with coefficients in $\mathbb{Z}_{\geq 0}[q]$. This sort of refined $\gamma$-nonnegativity also occurs in the ``cycle-type Eulerian polynomials" of Shareshian and Wachs \cite{SW10}. (The cycle-type Eulerian polynomials are generating functions for the joint distribution of major index and excedances over a set of permutations of fixed cycle type.) We conjecture a new, yet similar $\mathbb{Z}_{\geq 0}[q]$ $\gamma$-nonnegativity in Section \ref{sec:excinv}.

An idea going back to Biane \cite{Biane97} shows how the noncrossing partition lattice can be represented in a natural fashion as a certain partial order on a collection of Catalan-many permutations. This partial order on permutations, called the \emph{absolute order} generalizes to other Coxeter groups, and the interest in the lattice of noncrossing partitions now extends to the lattice of \emph{noncrossing partitions of type $W$}, where $W$ is a finite Coxeter group. A nice overview of the subject is presented in~\cite{A07}.

Thanks to work of Stump, \cite{Stump}, we are able to conclude that the area and rank of a Dyck path correspond to length, $\ell$, and reflection length, $\ell'$, for permutations in the interval $[e,(12\cdots n)]$ of the absolute order on $S_n$.

\begin{thm}[Stump \cite{Stump} Theorem 1.2 (1)]
There exists a bijection $\phi: \Dy(n) \to [e,(12\cdots n)]$ such that for any $p \in \Dy(n)$, \[ \area(p) = \ell( \phi(p)).\]
\end{thm}

As rank in $[e,(12\cdots n)]$ is given by reflection length and $\phi$ is rank-preserving, we immediately get the following.

\begin{cor}[Length and reflection length]\label{cor:ll'} 
For all $n\geq 1$,
\[ \Dy(n;q,t) = \sum_{\sigma \in [e,(12\cdots n)]} q^{\ell(\sigma)}t^{\ell'(\sigma)}.\]
\end{cor}
Thus, all the enumerative results we provide for area and rank of Dyck paths can be recast in terms of length and reflection length for these permutations. We remark that the joint distribution of length and reflection length in a Coxeter group has not been greatly investigated, though some attempt has been made at linking the two statistics. In particular, Edelman \cite{Edelman87} has studied the joint distribution of cycles with inversions in the symmetric group, obtaining some partial results. See also \cite{PetersenSorting}.

In Section \ref{sec:BKS}, we discuss two other bijections (due to Bandlow and Killpatrick \cite{BK01} and another due to Stump \cite{Stump}) for which area under a path corresponds to length of a permutation.

\begin{rmk}
Biane's correspondence works equally well for any choice of $n$-cycle $c$, i.e., as lattices $[e,c] \cong [e,(12\cdots n)]$. Since reflection length corresponds to rank, this implies that \[\sum_{\sigma \in [e,c]} t^{\ell'(\sigma)} = \sum_{\sigma \in [e,(12\cdots n)] } t^{\ell'(\sigma)}.\] However, it is \emph{not} generally true that length has the same distribution over two such intervals. In particular, the long element $\sigma_0 = n(n-1)\cdots 321$ (written in one-line notation) lies in $[e,(12\cdots n)]$ but not always in $[e,c]$. For example, $\sigma_0 = 4321$ does not lie in the interval $[e,(1243)]$. Thus, Corollary \ref{cor:ll'} is quite sensitive to our choice of $n$-cycle. 
\end{rmk}

We finish with Section \ref{sec:B}, in which we  show that the lattice of $B_n$-noncrossing partitions admits a symmetric boolean decomposition similar to Simion and Ullman's in the case of the symmetric group. (This was proved by Hersh using indirect means in \cite{H99}. An inductive proof is in \cite{PetersenShards}.) We ask whether the $W$-noncrossing partition lattice has such a decomposition for all $W$. Their rank generating functions are $\gamma$-nonnegative, so such a decomposition is possible.

\section{Area and rank of a Dyck path}\label{sec:stats}

In this section we establish a fundamental bijection between noncrossing partitions and Dyck paths that allows us to study the polynomials $\Dy(n;q,t)$.

\subsection{Noncrossing partitions and Dyck paths}\label{sec:noncrossingdyck}

Two pairs of numbers, say $\{a < b\}$ and $\{a' < b'\}$ are \emph{crossing} if either
\begin{itemize}
\item $a < a' < b < b'$ or
\item $a' < a < b' < b$.
\end{itemize}
More generally, if two disjoint sets $B$ and $B'$ contain pairs $\{ a < b\} \subseteq B$ and $\{a' < b'\} \subseteq B'$ that are crossing, we say $B$ and $B'$ are crossing. Otherwise, $B$ and $B'$ are \emph{noncrossing}. 

A \emph{noncrossing partition} of $[n]:=\{1,2,\ldots,n\}$, $\pi = \{ B_1, B_2,\ldots, B_k\}$ is a set partition of $[n]$ such that the blocks $B_i$ are pairwise noncrossing. We draw the \emph{arc diagram} of a noncrossing partition as a sequence of dots labeled $1, 2, \ldots, n$, with arcs connecting consecutive elements in the same block. For example, if $\pi = \{ \{1,8\}, \{2,4,7\}, \{3\}, \{5\}, \{6\} \}$, we draw 
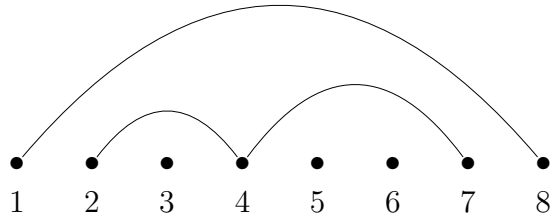
\[ \begin{xy}0;<1cm,0cm>:
(1,0)*{\bullet}="1", (2,0)*{\bullet}="2", (3,0)*{\bullet}="3", (4,0)*{\bullet}="4", (5,0)*{\bullet}="5", (6,0)*{\bullet}="6", (7,0)*{\bullet}="7", (8,0)*{\bullet}="8", (1,-.5)*{1}, (2,-.5)*{2}, (3,-.5)*{3}, (4,-.5)*{4}, (5,-.5)*{5}, (6,-.5)*{6}, (7,-.5)*{7}, (8,-.5)*{8},  {\ar @{-} @/^60pt/ "1"; "8" }, {\ar @{-} @/^20pt/ "2"; "4" }, {\ar @{-} @/^30pt/ "4"; "7" }
\end{xy}
\]
Let $NC(n)$ denote the set of all noncrossing partitions of $[n]$. We partially order this set by reverse refinement, so the set of all singletons is the minimum element, and $[n]$ itself is the maximum element. 

Our first task is to understand $(NC(n),\leq)$ in terms of Dyck paths.

\subsection{Stump's bijection $\phi$}

A \emph{Dyck path} of order $n$ (length $2n$) is a lattice path starting at $(0,0)$ and ending at $(2n,0)$, consisting of $n$ steps ``up" from $(i,j)$ to $(i+1,j+1)$ and $n$ steps ``down" from $(i,j)$ to $(i+1,j-1)$, such that the path never crosses below the $x$-axis. For example, the solid line in Figure \ref{fig:dyckex} is a Dyck path of rank $8$. (The numbers and other decorations will be explained shortly).
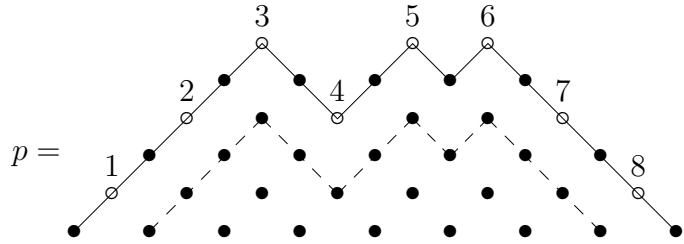
\begin{figure}[h]
\[
\begin{xy}0;<.5cm,0cm>:
(-1,2)*{p=}, (5,5); (0,0) **@{-}, (7,3) **@{-}, (9,5); (7,3) **@{-}, (10,4) **@{-}, (11,5); (10,4) **@{-}, (16,0) **@{-}, (5,3); (2,0) **@{--}, (7,1) **@{--},  (9,3); (7,1) **@{--}, (10,2) **@{--}, (11,3); (10,2) **@{--}, (14,0) **@{--}, (0,0)*{\bullet}, (1,1)*{\circ}, (2,2)*{\bullet}, (3,3)*{\circ}, (4,4)*{\bullet}, (5,5)*{\circ}, (6,4)*{\bullet}, (7,3)*{\circ}, (8,4)*{\bullet}, (9,5)*{\circ}, (10,4)*{\bullet}, (11,5)*{\circ}, (12,4)*{\bullet}, (13,3)*{\circ}, (14,2)*{\bullet}, (15,1)*{\circ}, (16,0)*{\bullet}, (2,0)*{\bullet}, (3,1)*{\bullet}, (4,2)*{\bullet}, (5,3)*{\bullet}, (4,0)*{\bullet}, (5,1)*{\bullet}, (6,2)*{\bullet}, (6,0)*{\bullet}, (7,1)*{\bullet}, (8,2)*{\bullet}, (9,3)*{\bullet}, (8,0)*{\bullet}, (9,1)*{\bullet}, (10,2)*{\bullet}, (11,3)*{\bullet}, (10,0)*{\bullet}, (11,1)*{\bullet}, (12,2)*{\bullet}, (12,0)*{\bullet}, (13,1)*{\bullet}, (14,0)*{\bullet}, (1,1.75)*{1}, (3,3.75)*{2}, (5,5.75)*{3}, (7,3.75)*{4}, (9,5.75)*{5}, (11,5.75)*{6}, (13,3.75)*{7}, (15,1.75)*{8}
\end{xy}
\]
\caption{A Dyck path $p$, with labels to illustrate the bijection $\phi: \Dy(n) \to NC(n)$.}\label{fig:dyckex}
\end{figure}

Let $\Dy(n)$ denote the set of all Dyck paths of order $n$. We now describe a bijection due to Stump \cite{Stump}, between $\Dy(n)$ and $NC(n)$, from which we will inherit a partial order on $\Dy(n)$. That is, we will declare $p \leq q$ in $\Dy(n)$ if and only if $\phi(p) \leq \phi(q)$ in $NC(n)$.

We now describe Stump's bijection $\phi: \Dy(n) \to NC(n)$. 

Consider the ``odd" points of the path $p$, i.e., those points of the form $(2i-1,2j-1)$ for $1\leq i,j \leq n$. In Figure \ref{fig:dyckex}, we see these vertices in white. We form the partition $\phi(p) = \pi \in NC(n)$ by declaring that $i$ and $i'$ are in the same block of $\pi$ if and only if: 
\begin{enumerate}
\item there exists a $j$ such that the points $(2i-1,2j-1)$ and $(2i'-1,2j-1)$ are on the path $p$, and
\item the segment of $p$ between $(2i-1,2j-1)$ and $(2i'-1,2j-1)$ does not pass below the line $y=2j-1$.
\end{enumerate}

For example, with the path $p$ drawn in Figure \ref{fig:dyckex}, we get \[\phi(p) = \{ \{1,8\}, \{2,4,7\}, \{3\}, \{5\}, \{6\}\}.\] Indeed, we see that the white vertices labeled 1 and 8 both sit at height $y=1$ and are connected by a segment of $p$ not passing below $y=1$. Similarly, 2, 4, and 7 are all at height $y=3$ and connected by a segment not passing below $y=3$. The vertices 3, 5, and 6 are at the same height, yet they are only connected by segments of $p$ that go below $y=5$. 

We now remark on another way to visualize the bijection $\phi$ via drawing a set of maximally nested Dyck paths below the path $p$.

For example, consider the path $p$ in Figure \ref{fig:dyckex}. Beneath this path we can draw another Dyck path that starts at $(2,0)$ and goes as close to the original Dyck path as possible. This is indicated with the dashed lines. We can continue to draw maximally nested Dyck paths (starting at the leftmost unused lattice point) until all the vertices below the original path have been used. Since there can be at most one maximal path beneath a given Dyck path starting at the leftmost unused lattice point, the collection of nested paths obtained from a given outer Dyck path is unique. Starting with the Dyck path in Figure \ref{fig:dyckex} we get the following collection of nested paths: 
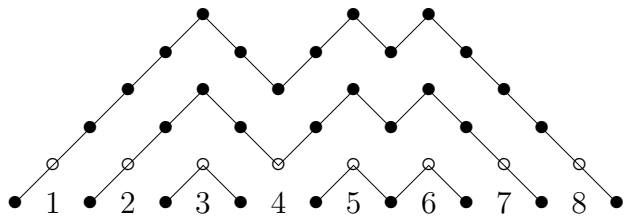
\[
\begin{xy}0;<.5cm,0cm>:
(5,5); (0,0) **@{-}, (7,3) **@{-}, (9,5); (7,3) **@{-}, (10,4) **@{-}, (11,5); (10,4) **@{-}, (16,0) **@{-}, (5,3); (2,0) **@{-}, (7,1) **@{-},  (9,3); (7,1) **@{-}, (10,2) **@{-}, (11,3); (10,2) **@{-}, (14,0) **@{-},
(5,1); (4,0) **@{-}, (6,0) **@{-}, (9,1); (8,0) **@{-}, (10,0) **@{-}, (11,1); (10,0) **@{-}, (12,0) **@{-}, (0,0)*{\bullet}, (1,1)*{\circ}, (2,2)*{\bullet}, (3,3)*{\bullet}, (4,4)*{\bullet}, (5,5)*{\bullet}, (6,4)*{\bullet}, (7,3)*{\bullet}, (8,4)*{\bullet}, (9,5)*{\bullet}, (10,4)*{\bullet}, (11,5)*{\bullet}, (12,4)*{\bullet}, (13,3)*{\bullet}, (14,2)*{\bullet}, (15,1)*{\circ}, (16,0)*{\bullet}, (2,0)*{\bullet}, (3,1)*{\circ}, (4,2)*{\bullet}, (5,3)*{\bullet}, (4,0)*{\bullet}, (5,1)*{\circ}, (6,2)*{\bullet}, (6,0)*{\bullet}, (7,1)*{\circ}, (8,2)*{\bullet}, (9,3)*{\bullet}, (8,0)*{\bullet}, (9,1)*{\circ}, (10,2)*{\bullet}, (11,3)*{\bullet}, (10,0)*{\bullet}, (11,1)*{\circ}, (12,2)*{\bullet}, (12,0)*{\bullet}, (13,1)*{\circ}, (14,0)*{\bullet},
(1,0)*{1}, (3,0)*{2}, (5,0)*{3}, (7,0)*{4}, (9,0)*{5}, (11,0)*{6}, (13,0)*{7}, (15,0)*{8} 
\end{xy}
\]
The paths that we have drawn are nothing but the segments of $p$ lying above heights $y=0$, $y=2$, $y=4$, and so on. In a picture of nested Dyck paths of order $n$ (i.e., with the outer path in $\Dy(n)$) there are $n$ lattice points of height 1: i.e., $(1,1), (3,1), \ldots, (2n-1,1)$. If we associate to $(2i-1)$ the element $i$, the nested Dyck paths give us a noncrossing partition of $n$. (The parts of the paths staying on or above level 1 are arcs in the arc diagram of the partition.) This process is easily reversed, so that, given a string diagram we can draw a set of nested Dyck paths. In Figure \ref{fig:NC4}, we have drawn the poset $(NC(4),\leq) \cong (\Dy(4),\leq)$. Notice that the rank of a noncrossing partition corresponds to $n$ minus the number of blocks, or to one less than the number of nested Dyck paths.

Let $\rk(p)$ denote the rank of the path $p$ in the lattice $(\Dy(n),\leq)$.

\begin{rmk}
There is a bijection $\alpha$ between Dyck paths of order $n$ and perfect \emph{noncrossing matchings} of $[2n]$ given by labeling the edges of a Dyck path with $\{1,2,\ldots,2n\}$ and then connecting $1\leq i,j\leq 2n$ if they can ``see'' each other horizontally (see~\cite[Section 2.1]{BDPR10}). Furthermore, there exists a bijection $\beta$ between noncrossing matchings of $[2n]$ and noncrossing partitions of $[n]$ given by identifying all pairs of consecutive elements $2i-1$ and $2i$ in a noncrossing matching. See, e.g.,~\cite{AST11, H06}. The bijection $\phi$ can be seen as the composition $\beta\circ\alpha$. We thank Drew Armstrong for bringing this fact to our attention. 
\end{rmk}

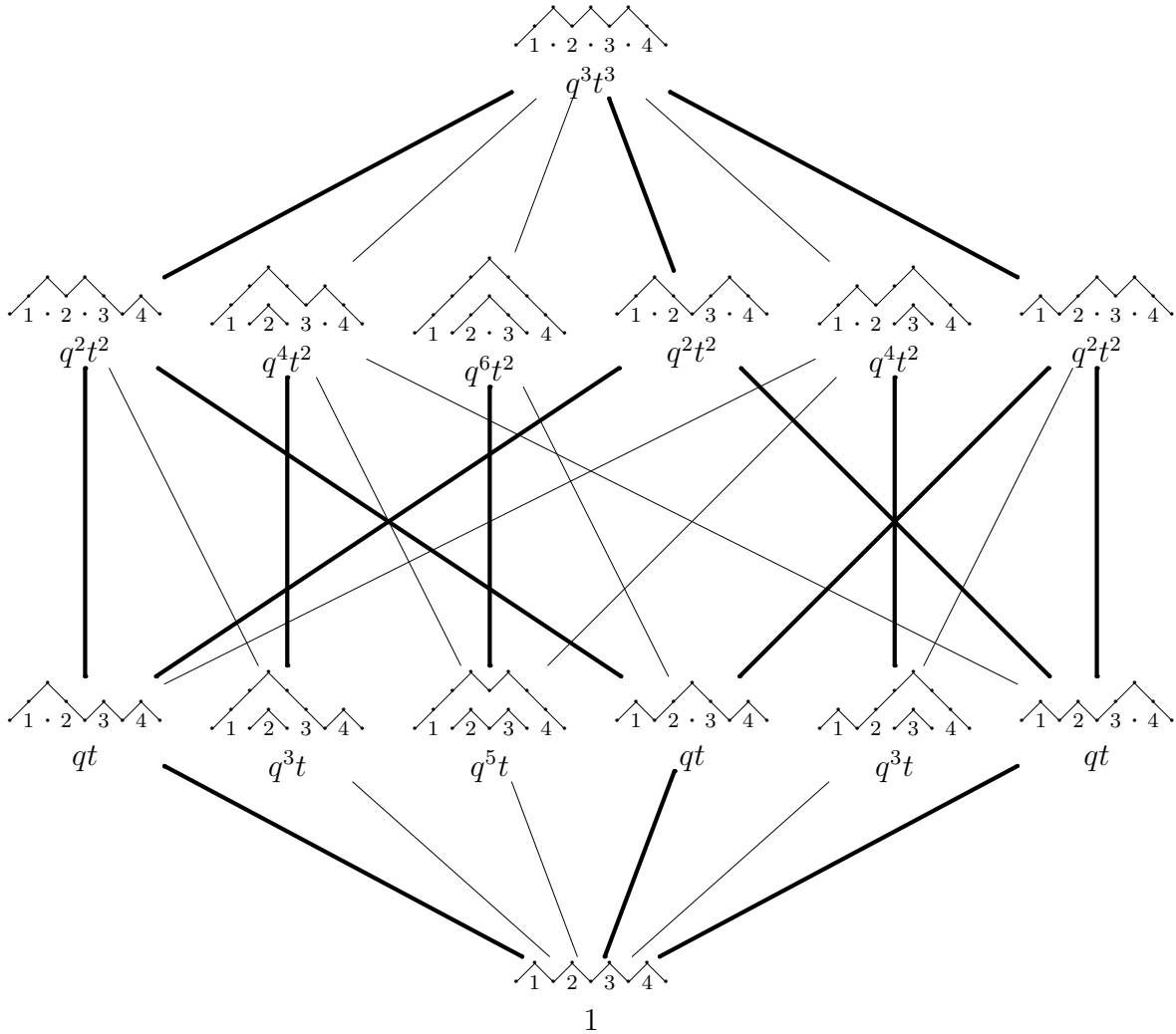
\begin{figure}[h]
\[
\begin{xy}0;<.9cm,0cm>:
(0,0)*{\begin{xy}0;<.25cm,0cm>:
(1,1); (0,0) **@{-}, (2,0) **@{-}, (3,1); (2,0) **@{-}, (4,0) **@{-}, (5,1); (4,0) **@{-}, (6,0) **@{-}, (7,1); (6,0) **@{-}, (8,0) **@{-}, 
(0,0)*{\cdot}, (2,0)*{\cdot}, (4,0)*{\cdot}, (6,0)*{\cdot}, (8,0)*{\cdot}, (1,1)*{\cdot}, (3,1)*{\cdot}, (5,1)*{\cdot}, (7,1)*{\cdot},
(1,0)*{\mbox{\tiny 1}}, (3,0)*{\mbox{\tiny 2}}, (5,0)*{\mbox{\tiny 3}}, (7,0)*{\mbox{\tiny 4}}, (4,-2)*{1}
\end{xy}}="a",
(-7.5,4)*{\begin{xy}0;<.25cm,0cm>:
(2,2); (0,0) **@{-}, (4,0) **@{-}, (5,1); (4,0) **@{-}, (6,0) **@{-}, (7,1); (6,0) **@{-}, (8,0) **@{-}, 
(0,0)*{\cdot}, (2,0)*{\cdot}, (4,0)*{\cdot}, (6,0)*{\cdot}, (8,0)*{\cdot}, (1,1)*{\cdot}, (3,1)*{\cdot}, (5,1)*{\cdot}, (7,1)*{\cdot}, (2,2)*{\cdot},
(1,0)*{\mbox{\tiny 1}}, (3,0)*{\mbox{\tiny 2}}, (5,0)*{\mbox{\tiny 3}}, (7,0)*{\mbox{\tiny 4}}, (4,-2)*{qt}
\end{xy}}="b1",
(-4.5,4)*{\begin{xy}0;<.25cm,0cm>:
(3,3); (0,0) **@{-}, (6,0) **@{-}, (3,1); (2,0) **@{-}, (4,0) **@{-}, (7,1); (6,0) **@{-}, (8,0) **@{-}, 
(0,0)*{\cdot}, (2,0)*{\cdot}, (4,0)*{\cdot}, (6,0)*{\cdot}, (8,0)*{\cdot}, (1,1)*{\cdot}, (3,1)*{\cdot}, (5,1)*{\cdot}, (7,1)*{\cdot}, (2,2)*{\cdot}, (4,2)*{\cdot}, (3,3)*{\cdot},
(1,0)*{\mbox{\tiny 1}}, (3,0)*{\mbox{\tiny 2}}, (5,0)*{\mbox{\tiny 3}}, (7,0)*{\mbox{\tiny 4}},
 (4,-2)*{q^3t}
\end{xy}}="b2",
(-1.5,4)*{\begin{xy}0;<.25cm,0cm>:
(3,3); (0,0) **@{-}, (4,2) **@{-}, (3,1); (2,0) **@{-}, (4,0) **@{-}, (5,1); (4,0) **@{-}, (6,0) **@{-}, (5,3); (4,2) **@{-}, (8,0) **@{-}, 
(0,0)*{\cdot}, (2,0)*{\cdot}, (4,0)*{\cdot}, (6,0)*{\cdot}, (8,0)*{\cdot}, (1,1)*{\cdot}, (3,1)*{\cdot}, (5,1)*{\cdot}, (7,1)*{\cdot}, (2,2)*{\cdot}, (4,2)*{\cdot}, (6,2)*{\cdot}, (3,3)*{\cdot}, (5,3)*{\cdot},
(1,0)*{\mbox{\tiny 1}}, (3,0)*{\mbox{\tiny 2}}, (5,0)*{\mbox{\tiny 3}}, (7,0)*{\mbox{\tiny 4}}, (4,-2)*{q^5t}
\end{xy}}="b3",
(1.5,4)*{\begin{xy}0;<.25cm,0cm>:
(1,1); (0,0) **@{-}, (2,0) **@{-}, (4,2); (2,0) **@{-}, (6,0) **@{-}, (7,1); (6,0) **@{-}, (8,0) **@{-}, 
(0,0)*{\cdot}, (2,0)*{\cdot}, (4,0)*{\cdot}, (6,0)*{\cdot}, (8,0)*{\cdot}, (1,1)*{\cdot}, (3,1)*{\cdot}, (5,1)*{\cdot}, (7,1)*{\cdot}, (4,2)*{\cdot},
(1,0)*{\mbox{\tiny 1}}, (3,0)*{\mbox{\tiny 2}}, (5,0)*{\mbox{\tiny 3}}, (7,0)*{\mbox{\tiny 4}}, (4,-2)*{qt}
\end{xy}}="b4",
(4.5,4)*{\begin{xy}0;<.25cm,0cm>:
(1,1); (0,0) **@{-}, (2,0) **@{-}, (5,3); (2,0) **@{-}, (8,0) **@{-}, (5,1); (4,0) **@{-}, (6,0) **@{-}, 
(0,0)*{\cdot}, (2,0)*{\cdot}, (4,0)*{\cdot}, (6,0)*{\cdot}, (8,0)*{\cdot}, (1,1)*{\cdot}, (3,1)*{\cdot}, (5,1)*{\cdot}, (7,1)*{\cdot},
(4,2)*{\cdot}, (5,3)*{\cdot}, (6,2)*{\cdot},
(1,0)*{\mbox{\tiny 1}}, (3,0)*{\mbox{\tiny 2}}, (5,0)*{\mbox{\tiny 3}}, (7,0)*{\mbox{\tiny 4}}, (4,-2)*{q^3t}
\end{xy}}="b5",
(7.5,4)*{\begin{xy}0;<.25cm,0cm>:
(1,1); (0,0) **@{-}, (2,0) **@{-}, (3,1); (2,0) **@{-}, (4,0) **@{-}, (6,2); (4,0) **@{-}, (8,0) **@{-}, 
(0,0)*{\cdot}, (2,0)*{\cdot}, (4,0)*{\cdot}, (6,0)*{\cdot}, (8,0)*{\cdot}, (1,1)*{\cdot}, (3,1)*{\cdot}, (5,1)*{\cdot}, (7,1)*{\cdot}, (6,2)*{\cdot},
(1,0)*{\mbox{\tiny 1}}, (3,0)*{\mbox{\tiny 2}}, (5,0)*{\mbox{\tiny 3}}, (7,0)*{\mbox{\tiny 4}}, (4,-2)*{qt}
\end{xy}}="b6",
(-7.5,10)*{\begin{xy}0;<.25cm,0cm>:
(2,2); (0,0) **@{-}, (3,1) **@{-}, (4,2); (3,1) **@{-}, (6,0) **@{-}, (7,1); (6,0) **@{-}, (8,0) **@{-}, 
(0,0)*{\cdot}, (2,0)*{\cdot}, (4,0)*{\cdot}, (6,0)*{\cdot}, (8,0)*{\cdot}, (1,1)*{\cdot}, (3,1)*{\cdot}, (5,1)*{\cdot}, (7,1)*{\cdot}, (2,2)*{\cdot}, (4,2)*{\cdot},
(1,0)*{\mbox{\tiny 1}}, (3,0)*{\mbox{\tiny 2}}, (5,0)*{\mbox{\tiny 3}}, (7,0)*{\mbox{\tiny 4}}, (4,-2)*{q^2t^2}
\end{xy}}="c1",
(4.5,10)*{\begin{xy}0;<.25cm,0cm>:
(2,2); (0,0) **@{-}, (3,1) **@{-}, (5,3); (3,1) **@{-}, (8,0) **@{-}, (5,1); (4,0) **@{-}, (6,0) **@{-}, 
(0,0)*{\cdot}, (2,0)*{\cdot}, (4,0)*{\cdot}, (6,0)*{\cdot}, (8,0)*{\cdot}, (1,1)*{\cdot}, (3,1)*{\cdot}, (5,1)*{\cdot}, (7,1)*{\cdot}, (2,2)*{\cdot}, (4,2)*{\cdot}, (5,3)*{\cdot},
(1,0)*{\mbox{\tiny 1}}, (3,0)*{\mbox{\tiny 2}}, (5,0)*{\mbox{\tiny 3}}, (7,0)*{\mbox{\tiny 4}}, (4,-2)*{q^4t^2}
\end{xy}}="c2",
(-4.5,10)*{\begin{xy}0;<.25cm,0cm>:
(3,3); (0,0) **@{-}, (5,1) **@{-}, (3,1); (2,0) **@{-}, (4,0) **@{-}, (6,2); (5,1) **@{-}, (8,0) **@{-},  
(0,0)*{\cdot}, (2,0)*{\cdot}, (4,0)*{\cdot}, (6,0)*{\cdot}, (8,0)*{\cdot}, (1,1)*{\cdot}, (3,1)*{\cdot}, (5,1)*{\cdot}, (7,1)*{\cdot}, (2,2)*{\cdot}, (4,2)*{\cdot}, (6,2)*{\cdot}, (3,3)*{\cdot}, 
(1,0)*{\mbox{\tiny 1}}, (3,0)*{\mbox{\tiny 2}}, (5,0)*{\mbox{\tiny 3}}, (7,0)*{\mbox{\tiny 4}}, (4,-2)*{q^4t^2}
\end{xy}}="c3",
(1.5,10)*{\begin{xy}0;<.25cm,0cm>:
(2,2); (0,0) **@{-}, (4,0) **@{-}, (6,2); (4,0) **@{-}, (8,0) **@{-}, (0,0)*{\cdot}, (2,0)*{\cdot}, (4,0)*{\cdot}, (6,0)*{\cdot}, (8,0)*{\cdot}, (1,1)*{\cdot}, (3,1)*{\cdot}, (5,1)*{\cdot}, (7,1)*{\cdot}, (2,2)*{\cdot}, (6,2)*{\cdot},
(1,0)*{\mbox{\tiny 1}}, (3,0)*{\mbox{\tiny 2}}, (5,0)*{\mbox{\tiny 3}}, (7,0)*{\mbox{\tiny 4}}, (4,-2)*{q^2t^2}
\end{xy}}="c4",
(-1.5,10)*{\begin{xy}0;<.25cm,0cm>:
(4,4); (0,0) **@{-}, (8,0) **@{-}, (4,2); (2,0) **@{-}, (6,0) **@{-}, (0,0)*{\cdot}, (2,0)*{\cdot}, (4,0)*{\cdot}, (6,0)*{\cdot}, (8,0)*{\cdot}, (1,1)*{\cdot}, (3,1)*{\cdot}, (5,1)*{\cdot}, (7,1)*{\cdot}, (4,4)*{\cdot},
(4,2)*{\cdot}, (5,3)*{\cdot}, (6,2)*{\cdot}, (3,3)*{\cdot}, (2,2)*{\cdot},
(1,0)*{\mbox{\tiny 1}}, (3,0)*{\mbox{\tiny 2}}, (5,0)*{\mbox{\tiny 3}}, (7,0)*{\mbox{\tiny 4}}, (4,-2)*{q^6t^2}
\end{xy}}="c5",
(7.5,10)*{\begin{xy}0;<.25cm,0cm>:
(1,1); (0,0) **@{-}, (2,0) **@{-}, (4,2); (2,0) **@{-}, (5,1) **@{-}, (6,2); (5,1) **@{-}, (8,0) **@{-}, 
(0,0)*{\cdot}, (2,0)*{\cdot}, (4,0)*{\cdot}, (6,0)*{\cdot}, (8,0)*{\cdot}, (1,1)*{\cdot}, (3,1)*{\cdot}, (5,1)*{\cdot}, (7,1)*{\cdot}, (6,2)*{\cdot}, (4,2)*{\cdot},
(1,0)*{\mbox{\tiny 1}}, (3,0)*{\mbox{\tiny 2}}, (5,0)*{\mbox{\tiny 3}}, (7,0)*{\mbox{\tiny 4}}, (4,-2)*{q^2t^2}
\end{xy}}="c6",
(0,14)*{\begin{xy}0;<.25cm,0cm>:
(2,2); (0,0) **@{-}, (3,1) **@{-}, (4,2); (3,1) **@{-}, (5,1) **@{-}, (6,2); (5,1) **@{-}, (8,0) **@{-}, 
(0,0)*{\cdot}, (2,0)*{\cdot}, (4,0)*{\cdot}, (6,0)*{\cdot}, (8,0)*{\cdot}, (1,1)*{\cdot}, (3,1)*{\cdot}, (5,1)*{\cdot}, (7,1)*{\cdot}, (6,2)*{\cdot}, (2,2)*{\cdot}, (4,2)*{\cdot},
(1,0)*{\mbox{\tiny 1}}, (3,0)*{\mbox{\tiny 2}}, (5,0)*{\mbox{\tiny 3}}, (7,0)*{\mbox{\tiny 4}}, (4,-2)*{q^3t^3}
\end{xy}}="d",
"a"; {\ar @[|(4)] @{-}, "a"; "b1"}, "b2" **@{-}, "b3" **@{-}, {\ar @[|(4)] @{-}, "a"; "b4"}, "b5" **@{-}, {\ar @[|(4)] @{-}, "a"; "b6"},
"b1"; {\ar @[|(4)] @{-}, "b1"; "c1"}, "c2" **@{-},  {\ar @[|(4)] @{-}, "b1"; "c4"}, 
"b2"; "c1" **@{-}, {\ar @[|(4)] @{-}, "b2"; "c3"},
"b3"; "c2" **@{-}, "c3" **@{-}, {\ar @[|(4)] @{-}, "b3"; "c5"},
"b4"; {\ar @[|(4)] @{-}, "b4"; "c1"}, "c5" **@{-}, {\ar @[|(4)] @{-}, "b4"; "c6"},
"b5"; {\ar @[|(4)] @{-}, "b5"; "c2"}, "c6" **@{-},
"b6"; "c3" **@{-}, {\ar @[|(4)] @{-}, "b6"; "c4"}, {\ar @[|(4)] @{-}, "b6"; "c6"},
"d"; {\ar @[|(4)] @{-}, "d"; "c1"}, {\ar @[|(4)] @{-}, "d"; "c4"}, {\ar @[|(4)] @{-}, "d"; "c6"}, "c2" **@{-}, "c3" **@{-}, "c5" **@{-}
\end{xy}
\]
\caption{The noncrossing partitions of $\{1,2,3,4\}$, drawn as nested Dyck paths. Weights are given with respect to area and rank. Boldface lines indicate Simion and Ullmann's symmetric boolean decomposition of the poset.}\label{fig:NC4}
\end{figure}

Apart from the rank of a Dyck path in $(\Dy(n),\leq)$, we will keep track of another statistic. We define the \emph{area of $p$}, denoted $\area(p)$, to be the number of unused lattice points below the path and on or above the line $y=0$. So, for example, we have $\area(p) = 8$ for the following path:
 
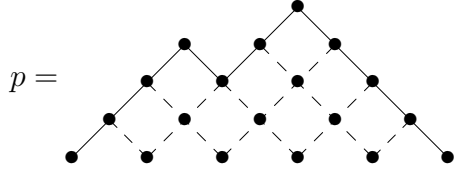
\[ \begin{xy}0;<.5cm,0cm>:
(-1,2)*{p=}; (3,3); (0,0) **@{-}, (4,2) **@{-}, (6,4); (4,2) **@{-}, (10,0) **@{-}, (2,0); (1,1) **@{--}, (4,2) **@{--}, (4,0); (2,2) **@{--}, (7,3) **@{--}, (6,0); (4,2) **@{--}, (8,2) **@{--}, (8,0); (5,3) **@{--}, (9,1) **@{--}, (0,0)*{\bullet}, (1,1)*{\bullet}, (2,2)*{\bullet}, (3,3)*{\bullet},  (6,4)*{\bullet}, (2,0)*{\bullet}, (3,1)*{\bullet}, (4,2)*{\bullet}, (5,3)*{\bullet}, (4,0)*{\bullet}, (5,1)*{\bullet}, (6,2)*{\bullet}, (6,0)*{\bullet}, (7,1)*{\bullet}, (8,2)*{\bullet}, (8,0)*{\bullet}, (9,1)*{\bullet}, (10,0)*{\bullet}, (7,3)*{\bullet}
\end{xy}
\]
We use the term ``area" because this is the number of $1/\sqrt{2}\times 1/\sqrt{2}$ diamonds that can fit below the path and above the $x$-axis. (Just put the bottom of the diamonds on the unused lattice points.) This is indicated with dashed lines in the picture above. The distribution of area for Dyck paths has has been studied by many, starting with Carlitz and Riordan \cite{CR64}.

Define the following bivariate generating function for area and rank of Dyck paths:
\[ \Dy(n;q,t) = \sum_{p \in \Dy(n)} q^{\area(p)}t^{\rk(p)}.\]

For example, with $n=4$ and $n=5$ we have:
\begin{align} 
 \Dy(4;q,t) &=  1 + 3qt + 2q^3t + q^5t + 3q^2t^2 + 2q^4t^2 + q^6t^2 + q^3t^3 \nonumber \\
 &= (1+qt)^{3}+2q^{3}t(1+qt)+q^{5}t(1+qt)\notag\\
 &= (1+qt)^3 + (2q^3+q^5)t(1+qt) \label{eq:D4}\\
 \Dy(5;q,t) &=1 + 4qt + 6q^2t^2 + 3q^3t + 2q^5t + q^7t + 6q^4t^2 + 5q^6t^2 \nonumber\\
 &\qquad + q^8t^2 + q^{10}t^2 + 4q^3t^3 + 3q^5t^3 + 2q^7t^3 + q^9t^3+ q^4t^4  \nonumber \\
 &=(1+qt)^{4}+3q^{3}t(1+qt)^{2} + 2q^5t(1+qt)^2 + q^7t(1+qt)^2 + q^6t^2 + q^{10}t^2 \notag \\
 &=  (1+qt)^4 + (3q^3 + 2q^5+q^7)t(1+qt)^2 + (q^6 + q^{10})t^2 \label{eq:D5}
  \end{align}

The lovely symmetry of these polynomials is perhaps better seen in the following arrays of their coefficients (with $\cdot$ instead of 0):
\[ n=4: \begin{array}{c | c c c c c c c}
t \backslash q & 0 & 1 & 2 & 3 & 4 & 5 & 6 \\
\hline 
0 & 1 & \cdot & \cdot & \cdot & \cdot & \cdot & \cdot \\
1 & \cdot & 3 & \cdot & 2 & \cdot & 1 & \cdot \\
2 & \cdot & \cdot & 3 & \cdot & 2 & \cdot & 1 \\
3 & \cdot & \cdot & \cdot & 1 & \cdot & \cdot & \cdot
\end{array}
\qquad n=5:  \begin{array}{c | c c c c c c c c c c c }
 t \backslash q & 0 & 1 & 2 & 3 & 4 & 5 & 6 & 7 & 8 & 9 & 10\\
\hline 
0 & 1 & \cdot & \cdot & \cdot & \cdot & \cdot & \cdot & \cdot & \cdot & \cdot & \cdot \\
1 & \cdot & 4 & \cdot & 3 & \cdot & 2 &\cdot & 1 & \cdot & \cdot & \cdot\\
2 & \cdot & \cdot & 6 & \cdot & 6 & \cdot & 5 & \cdot & 2 & \cdot & 1 \\
3 & \cdot & \cdot & \cdot & 4 & \cdot & 3 & \cdot & 2 & \cdot & 1 & \cdot\\
4 & \cdot & \cdot & \cdot & \cdot & 1 & \cdot & \cdot & \cdot & \cdot & \cdot & \cdot
\end{array}
\]

We will now discuss Motzkin paths and explain why we are able to express polynomials $\Dy(n;q,t)$ as $\mathbb{Z}[q]$-positive sums of  terms of the form $t^j(1+qt)^{n-1-2j}$.

\subsection{Motzkin paths}

Let $\M(n)$ denote the set of \emph{Motzkin paths} of order $n$. These are paths from $(1,1)$ to $(2n-1,1)$ that never pass below the line $y=1$, while taking steps ``up" from $(i,j)$ to $(i+2,j+2)$, ``down" from $(i,j)$ to $(i+2,j-2)$, or ``across" from $(i,j)$ to $(i+2,j)$. While the choice of starting point, ending point and step size is not standard (cf.~\cite[Exercise 6.38(d)]{EC2}), they suit our purposes better.

For example, when $n=5$, there are nine such paths, as seen in Figure \ref{fig:Motz}. 

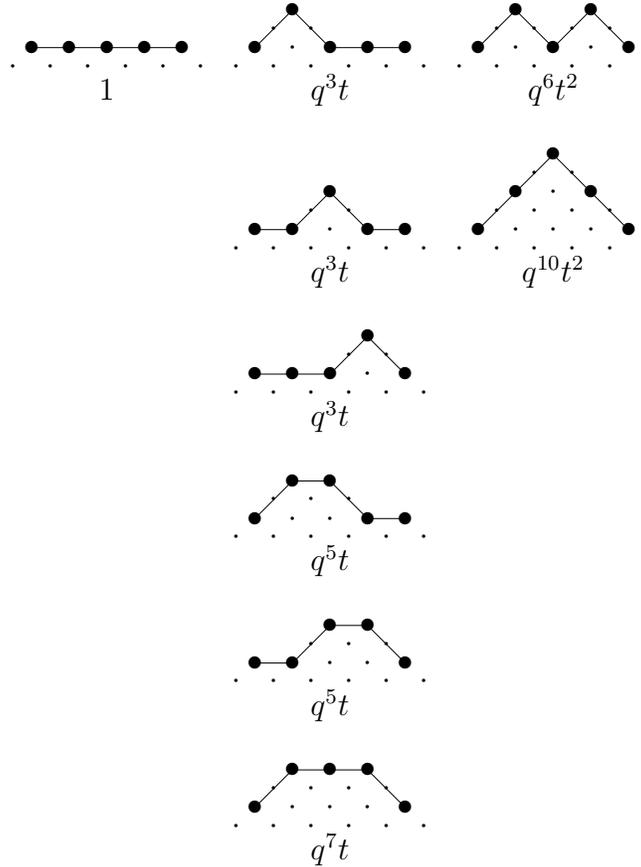
\begin{figure}[h]
\[
\begin{array}{ c c c }
\begin{xy}0;<.25cm,0cm>:
(0,0); (8,0) **@{-}, (0,0)*{\bullet}, (2,0)*{\bullet},  (4,0)*{\bullet}, (6,0)*{\bullet},  (8,0)*{\bullet}, (-1,-1)*{\cdot}, (1,-1)*{\cdot}, (3,-1)*{\cdot}, (5,-1)*{\cdot}, (7,-1)*{\cdot}, (9,-1)*{\cdot}
\end{xy} 
&
\begin{xy}0;<.25cm,0cm>:
(0,0); (2,2) **@{-}, (2,2); (4,0) **@{-}, (4,0); (8,0) **@{-}, (0,0)*{\bullet},  (2,0)*{\cdot},  (4,0)*{\bullet}, (6,0)*{\bullet}, (8,0)*{\bullet}, (2,2)*{\bullet}, (-1,-1)*{\cdot}, (1,-1)*{\cdot}, (3,-1)*{\cdot}, (5,-1)*{\cdot}, (7,-1)*{\cdot}, (9,-1)*{\cdot}, (1,1)*{\cdot}, (3,1)*{\cdot}
\end{xy}
&
\begin{xy}0;<.25cm,0cm>:
(0,0); (2,2) **@{-}, (2,2); (4,0) **@{-}, (4,0); (6,2) **@{-}, (6,2); (8,0) **@{-}, (0,0)*{\bullet}, (2,0)*{\cdot}, (4,0)*{\bullet}, (6,0)*{\cdot}, (8,0)*{\bullet}, (6,2)*{\bullet}, (2,2)*{\bullet}, (-1,-1)*{\cdot}, (1,-1)*{\cdot}, (3,-1)*{\cdot}, (5,-1)*{\cdot}, (7,-1)*{\cdot}, (9,-1)*{\cdot}, (3,1)*{\cdot}, (5,1)*{\cdot}, (1,1)*{\cdot}, (7,1)*{\cdot}
\end{xy}
\\ 1 & q^3t & q^6t^2 \\ \\
&
\begin{xy}0;<.25cm,0cm>:
(0,0); (2,0) **@{-}, (2,0); (4,2) **@{-}, (4,2); (6,0) **@{-}, (6,0); (8,0) **@{-}, (0,0)*{\bullet},(2,0)*{\bullet}, (4,0)*{\cdot}, (6,0)*{\bullet}, (8,0)*{\bullet}, (4,2)*{\bullet}, (-1,-1)*{\cdot}, (1,-1)*{\cdot}, (3,-1)*{\cdot}, (5,-1)*{\cdot}, (7,-1)*{\cdot}, (9,-1)*{\cdot}, (3,1)*{\cdot}, (5,1)*{\cdot}
\end{xy}
&
\begin{xy}0;<.25cm,0cm>:
(0,0); (4,4) **@{-}, (4,4); (8,0) **@{-}, (0,0)*{\bullet},(2,0)*{\cdot}, (4,0)*{\cdot}, (6,0)*{\cdot}, (8,0)*{\bullet}, (6,2)*{\bullet}, (4,2)*{\cdot}, (2,2)*{\bullet}, (4,4)*{\bullet}, (-1,-1)*{\cdot}, (1,-1)*{\cdot}, (3,-1)*{\cdot}, (5,-1)*{\cdot}, (7,-1)*{\cdot}, (9,-1)*{\cdot}, (3,1)*{\cdot}, (5,1)*{\cdot}, (3,3)*{\cdot}, (5,3)*{\cdot}, (7,1)*{\cdot}, (1,1)*{\cdot}
\end{xy}
\\ & q^3t & q^{10}t^2 \\ \\
&
\begin{xy}0;<.25cm,0cm>:
(0,0); (4,0) **@{-}, (4,0); (6,2) **@{-}, (6,2); (8,0) **@{-}, (0,0)*{\bullet},(2,0)*{\bullet}, (4,0)*{\bullet}, (6,0)*{\cdot}, (8,0)*{\bullet}, (6,2)*{\bullet}, (-1,-1)*{\cdot}, (1,-1)*{\cdot}, (3,-1)*{\cdot}, (5,-1)*{\cdot}, (7,-1)*{\cdot}, (9,-1)*{\cdot}, (5,1)*{\cdot}, (7,1)*{\cdot}
\end{xy}
\\ & q^3t & \\ \\
& \begin{xy}0;<.25cm,0cm>:
(0,0); (2,2) **@{-}, (2,2); (4,2) **@{-}, (4,2); (6,0) **@{-}, (6,0); (8,0) **@{-}, (0,0)*{\bullet},(2,0)*{\cdot}, (4,0)*{\cdot}, (6,0)*{\bullet}, (8,0)*{\bullet}, (2,2)*{\bullet}, (4,2)*{\bullet}, (-1,-1)*{\cdot}, (1,-1)*{\cdot}, (3,-1)*{\cdot}, (5,-1)*{\cdot}, (7,-1)*{\cdot}, (9,-1)*{\cdot}, (1,1)*{\cdot}, (3,1)*{\cdot}, (5,1)*{\cdot}
\end{xy} 
\\ & q^5t \\ \\
&
\begin{xy}0;<.25cm,0cm>:
(0,0); (2,0) **@{-}, (2,0); (4,2) **@{-}, (4,2); (6,2) **@{-}, (6,2); (8,0) **@{-}, (0,0)*{\bullet},(2,0)*{\bullet}, (4,0)*{\cdot}, (6,0)*{\cdot}, (8,0)*{\bullet}, (4,2)*{\bullet}, (6,2)*{\bullet}, (-1,-1)*{\cdot}, (1,-1)*{\cdot}, (3,-1)*{\cdot}, (5,-1)*{\cdot}, (7,-1)*{\cdot}, (9,-1)*{\cdot}, (3,1)*{\cdot}, (5,1)*{\cdot}, (7,1)*{\cdot}
\end{xy}
\\ & q^5t\\ \\
&
\begin{xy}0;<.25cm,0cm>:
(0,0); (2,2) **@{-}, (2,2); (6,2) **@{-}, (6,2); (8,0) **@{-}, (0,0)*{\bullet},(2,0)*{\cdot}, (4,0)*{\cdot}, (6,0)*{\cdot}, (8,0)*{\bullet}, (6,2)*{\bullet}, (4,2)*{\bullet}, (2,2)*{\bullet}, (-1,-1)*{\cdot}, (1,-1)*{\cdot}, (3,-1)*{\cdot}, (5,-1)*{\cdot}, (7,-1)*{\cdot}, (9,-1)*{\cdot}, (1,1)*{\cdot}, (3,1)*{\cdot}, (5,1)*{\cdot}, (7,1)*{\cdot}
\end{xy}
\\
& q^7t
 \end{array}
\]
\caption{Motzkin paths of order 5.}\label{fig:Motz}
\end{figure}

We will define an injection $\rho: \M(n) \to \Dy(n)$. For a Motzkin path $m\in \M(n)$, let $\rho(m)=p \in \Dy(n)$ denote the Dyck path obtained by adding an up step on the left, adding a down step on the right, and replacing each horizontal step with a down-up step. Moreover, up steps in $m$ become two up steps in $p$ and down steps in $m$ become two down steps in $p$:
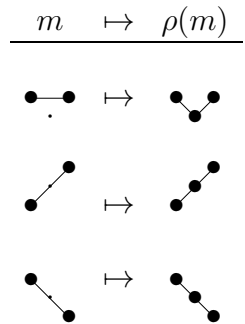
\[ \begin{array}{ c c c } m & \mapsto & \rho(m) \\ \hline  \\
 \begin{xy}0;<.5cm,0cm>:
(0,0); (1,0) **@{-}, (0,0)*{\bullet}, (1,0)*{\bullet}, (.5,-.5)*{\cdot}\end{xy} & \mapsto & \begin{xy}0;<.5cm,0cm>:
(.5,-.5); (0,0) **@{-}, (1,0) **@{-}, (0,0)*{\bullet}, (1,0)*{\bullet}, (.5,-.5)*{\bullet}
\end{xy} \\ \\
\begin{xy}0;<.5cm,0cm>:
(0,0); (1,1) **@{-}, (0,0)*{\bullet}, (1,1)*{\bullet}, (.5,.5)*{\cdot}
\end{xy} & \mapsto & \begin{xy}0;<.5cm,0cm>:
(0,0); (1,1) **@{-}, (0,0)*{\bullet}, (1,1)*{\bullet}, (.5,.5)*{\bullet}
\end{xy} \\ \\
\begin{xy}0;<.5cm,0cm>:
(0,0); (1,-1) **@{-}, (0,0)*{\bullet}, (1,-1)*{\bullet}, (.5,-.5)*{\cdot}
\end{xy} & \mapsto & \begin{xy}0;<.5cm,0cm>:
(0,0); (1,-1) **@{-}, (0,0)*{\bullet}, (1,-1)*{\bullet}, (.5,-.5)*{\bullet}
\end{xy}
 \end{array}\]
For example, below is a Motzkin path $m$ and its corresponding Dyck path $p$:
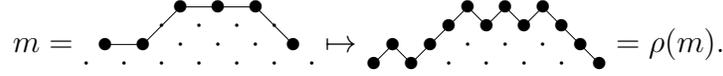
\[
m = \begin{xy}0;<.25cm,0cm>:
(0,0);  (2,0) **@{-}, (2,0); (4,2) **@{-}, (4,2); (8,2) **@{-}, (8,2); (10,0) **@{-}, (0,0)*{\bullet},(2,0)*{\bullet}, (4,0)*{\cdot}, (6,0)*{\cdot}, (10,0)*{\bullet}, (4,2)*{\bullet}, (6,2)*{\bullet}, (8,2)*{\bullet}, (-1,-1)*{\cdot}, (1,-1)*{\cdot}, (3,-1)*{\cdot}, (5,-1)*{\cdot}, (7,-1)*{\cdot}, (9,-1)*{\cdot}, (3,1)*{\cdot}, (5,1)*{\cdot}, (7,1)*{\cdot}, (9,1)*{\cdot}, (8,0)*{\cdot}, (11,-1)*{\cdot}
\end{xy} \mapsto \begin{xy}0;<.25cm,0cm>:
(0,0); (-1,-1) **@{-},(1,-1) **@{-}, (1,-1); (4,2) **@{-}, (5,1); (4,2) **@{-}, (6,2) **@{-}, (7,1); (6,2) **@{-}, (8,2) **@{-}, (8,2); (11,-1) **@{-}, (0,0)*{\bullet},(2,0)*{\bullet}, (4,0)*{\cdot}, (6,0)*{\cdot}, (10,0)*{\bullet}, (4,2)*{\bullet}, (6,2)*{\bullet}, (8,2)*{\bullet}, (-1,-1)*{\bullet}, (1,-1)*{\bullet}, (3,-1)*{\cdot}, (5,-1)*{\cdot}, (7,-1)*{\cdot}, (9,-1)*{\cdot}, (3,1)*{\bullet}, (5,1)*{\bullet}, (7,1)*{\bullet}, (9,1)*{\bullet}, (8,0)*{\cdot}, (11,-1)*{\bullet}
\end{xy} = \rho(m).
\]

We define the area and rank of a Motzkin path $m$ to be the area and rank of $\rho(m)$: \[ \area(m) = \area(\rho(m)), \qquad \rk(m) = \rk(\rho(m)).\] Put simply, $\area(m)$ is the number of diamonds that will fit below the path just as with Dyck paths, and $\rk(m)$ is the number of ``up" steps in $m$. Thus in the example above we have $\area(m) = 7$, $\rk(m) = 1$. Denote the generating function for area and rank of Motzkin paths as: \[ \M(n;q,t) = \sum_{m \in \M(n)} q^{\area(m)} t^{\rk(m)}.\]
For example, when $n =4$ and $n=5$ we have:
\begin{align*}
\M(4;q,t) &= 1 + (2q^3 + q^5)t,\\
\M(5;q,t) &= 1 + (3q^3 + 2q^5 + q^7)t + (q^6 + q^{10})t^2.
\end{align*}

Comparing with the generating functions for Dyck paths in \eqref{eq:D4} and \eqref{eq:D5}, we see evidence for Theorem \ref{thm:main}. 

We will now consider a surjection $\theta: \Dy(n) \to \M(n)$. Given a Dyck path $p$, we form $\theta(p) = m \in \M(n)$ by connecting consecutive odd nodes of $p$ (i.e., those nodes of the form $(2i-1,2j-1)$), with straight lines. For example, below is a Dyck path and its corresponding Motzkin path:
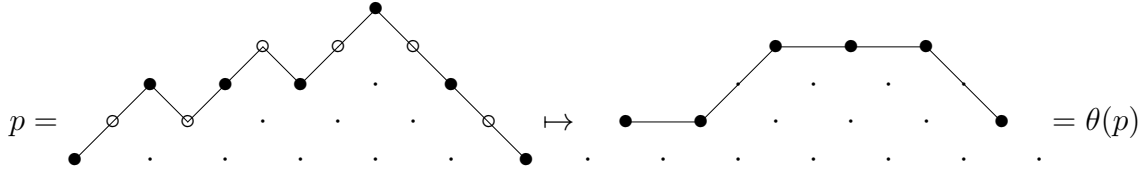
\[
p=\begin{xy}0;<.5cm,0cm>:
(1,1); (-1,-1) **@{-},(2,0) **@{-}, (2,0); (4,2) **@{-}, (5,1); (4,2) **@{-}, (6,2) **@{-}, (7,3); (6,2) **@{-}, (8,2) **@{-}, (8,2); (11,-1) **@{-}, (0,0)*{\circ}, (2,0)*{\circ}, (4,0)*{\cdot}, (6,0)*{\cdot}, (10,0)*{\circ}, (4,2)*{\circ}, (6,2)*{\circ}, (8,2)*{\circ}, (-1,-1)*{\bullet}, (1,1)*{\bullet}, (1,-1)*{\cdot}, (3,-1)*{\cdot}, (5,-1)*{\cdot}, (7,-1)*{\cdot}, (9,-1)*{\cdot}, (3,1)*{\bullet}, (5,1)*{\bullet}, (7,1)*{\cdot}, (7,3)*{\bullet}, (9,1)*{\bullet}, (8,0)*{\cdot}, (11,-1)*{\bullet}
\end{xy} \mapsto
\begin{xy}0;<.5cm,0cm>:
(0,0);  (2,0) **@{-}, (2,0); (4,2) **@{-}, (4,2); (8,2) **@{-}, (8,2); (10,0) **@{-}, (0,0)*{\bullet},(2,0)*{\bullet}, (4,0)*{\cdot}, (6,0)*{\cdot}, (10,0)*{\bullet}, (4,2)*{\bullet}, (6,2)*{\bullet}, (8,2)*{\bullet}, (-1,-1)*{\cdot}, (1,-1)*{\cdot}, (3,-1)*{\cdot}, (5,-1)*{\cdot}, (7,-1)*{\cdot}, (9,-1)*{\cdot}, (3,1)*{\cdot}, (5,1)*{\cdot}, (7,1)*{\cdot}, (9,1)*{\cdot}, (8,0)*{\cdot}, (11,-1)*{\cdot}
\end{xy} = \theta(p)
\]

Notice that the fibers of a given Motzkin path $m$ all have their odd (white) vertices in the same locations. Thus, they can only differ between  these nodes. In particular, suppose $p\neq p' \in \Dy(n)$ are such that $\theta(p) = \theta(p')$. Then there are two consecutive white vertices such that, say, $p$ goes up-down between them: \begin{xy}0;<.25cm,0cm>:
(1,1)*{\bullet}; (0,0)*{\circ} **@{-}, (2,0)*{\circ} **@{-}
\end{xy},
while $p'$ goes down-up: \begin{xy}0;<.25cm,0cm>:
(1,0)*{\bullet}; (0,1)*{\circ} **@{-}, (2,1)*{\circ} **@{-}
\end{xy}.

For a given Motzkin path $m \in \M(n)$, let \[\B(m) = \{ p \in \Dy(n) : \theta(p) = m\}.\] For example in Figure \ref{fig:thetam} we see the set $\B(m)$ for a particular $m$. Notice that the partial order on $\B(m)$ induced by $(\Dy(n),\leq)$ is the boolean algebra given by flipping down-up steps between odd vertices for up-down steps. The minimal element of $\B(m)$ is thus $\rho(m)$.

We will make this idea precise by relating our constructions to work of Simion and Ullman.

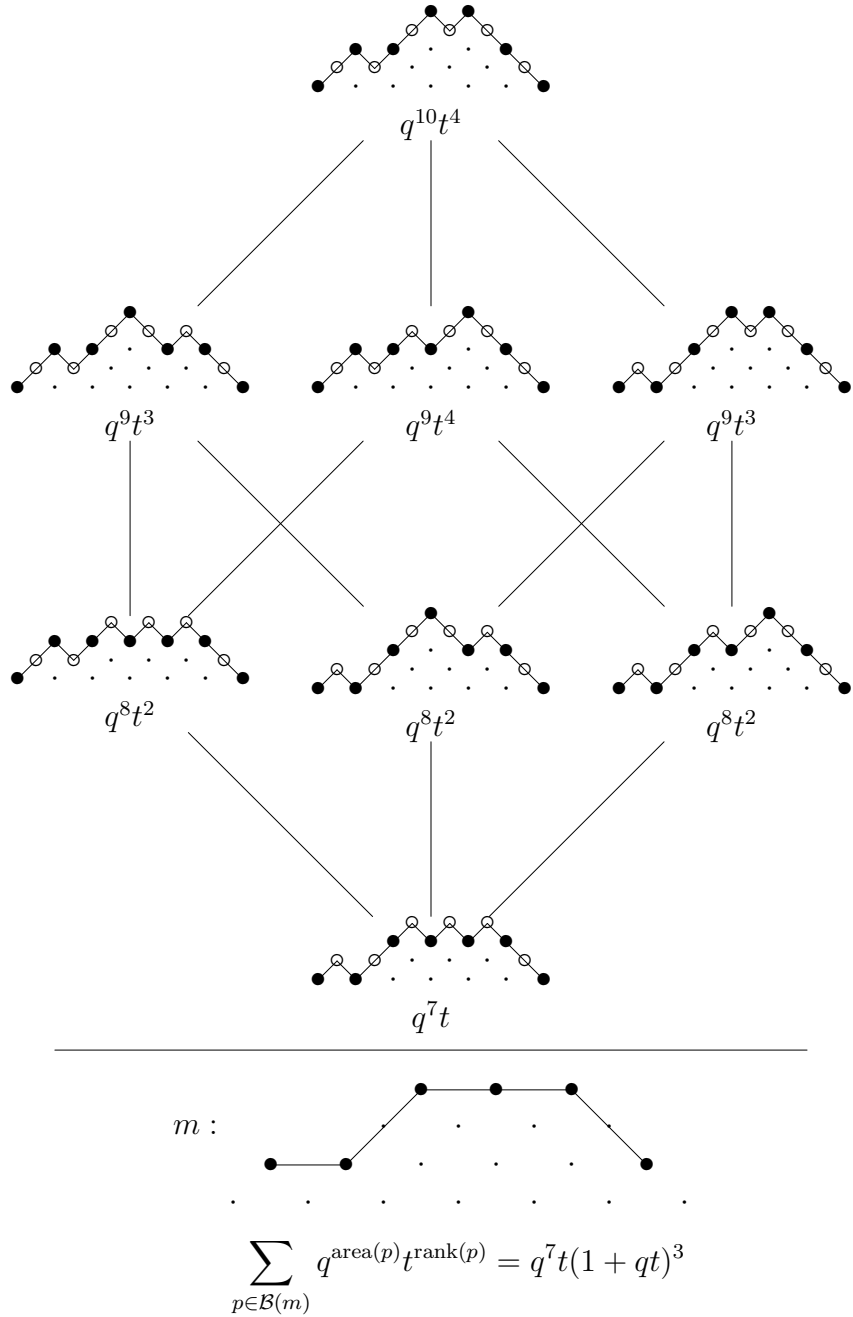
\begin{figure}
\[
\begin{xy}0;<1cm,0cm>:
(-5,2); (5,2) **@{-},
(0,0)*{\begin{xy}0;<.5cm,0cm>:
(0,0);  (2,0) **@{-}, (2,0); (4,2) **@{-}, (4,2); (8,2) **@{-}, (8,2); (10,0) **@{-}, (0,0)*{\bullet},(2,0)*{\bullet}, (4,0)*{\cdot}, (6,0)*{\cdot}, (10,0)*{\bullet}, (4,2)*{\bullet}, (6,2)*{\bullet}, (8,2)*{\bullet}, (-1,-1)*{\cdot}, (1,-1)*{\cdot}, (3,-1)*{\cdot}, (5,-1)*{\cdot}, (7,-1)*{\cdot}, (9,-1)*{\cdot}, (3,1)*{\cdot}, (5,1)*{\cdot}, (7,1)*{\cdot}, (9,1)*{\cdot}, (8,0)*{\cdot}, (11,-1)*{\cdot}, (-2,1)*{m:},(5,-3)*{\displaystyle\sum_{p \in \B(m)} q^{\area(p)}t^{\rk(p)} = q^7t(1+qt)^3}
\end{xy}}="m",
(0,3)*{\begin{xy}0;<.25cm,0cm>:
(0,0); (-1,-1) **@{-},(1,-1) **@{-}, (4,2); (1,-1) **@{-}, (5,1) **@{-}, (6,2); (5,1) **@{-},  (7,1) **@{-}, (8,2); (7,1) **@{-}, (11,-1) **@{-}, (0,0)*{\circ}, (2,0)*{\circ}, (4,0)*{\cdot}, (6,0)*{\cdot}, (10,0)*{\circ}, (4,2)*{\circ}, (6,2)*{\circ}, (8,2)*{\circ}, (-1,-1)*{\bullet},  (1,-1)*{\bullet}, (3,-1)*{\cdot}, (5,-1)*{\cdot}, (7,-1)*{\cdot}, (9,-1)*{\cdot}, (3,1)*{\bullet}, (5,1)*{\bullet}, (7,1)*{\bullet}, (9,1)*{\bullet}, (8,0)*{\cdot}, (11,-1)*{\bullet}, (5,-3)*{q^{7}t}
\end{xy}}="p0",
(-4,7)*{\begin{xy}0;<.25cm,0cm>:
(1,1); (-1,-1) **@{-},(2,0) **@{-}, (4,2); (2,0) **@{-}, (5,1) **@{-}, (6,2); (5,1) **@{-},  (7,1) **@{-}, (8,2); (7,1) **@{-}, (11,-1) **@{-}, (0,0)*{\circ}, (2,0)*{\circ}, (4,0)*{\cdot}, (6,0)*{\cdot}, (10,0)*{\circ}, (4,2)*{\circ}, (6,2)*{\circ}, (8,2)*{\circ}, (-1,-1)*{\bullet},  (1,-1)*{\cdot}, (1,1)*{\bullet}, (3,-1)*{\cdot}, (5,-1)*{\cdot}, (7,-1)*{\cdot}, (9,-1)*{\cdot}, (3,1)*{\bullet}, (5,1)*{\bullet}, (7,1)*{\bullet}, (9,1)*{\bullet}, (8,0)*{\cdot}, (11,-1)*{\bullet}, (5,-3)*{q^{8}t^2}
\end{xy}}="p1",
(0,7)*{\begin{xy}0;<.25cm,0cm>:
(0,0); (-1,-1) **@{-},(1,-1) **@{-}, (5,3); (1,-1) **@{-}, (7,1) **@{-}, (8,2); (7,1) **@{-}, (11,-1) **@{-}, (0,0)*{\circ}, (2,0)*{\circ}, (4,0)*{\cdot}, (6,0)*{\cdot}, (10,0)*{\circ}, (4,2)*{\circ}, (6,2)*{\circ}, (8,2)*{\circ}, (-1,-1)*{\bullet},  (1,-1)*{\bullet}, (3,-1)*{\cdot}, (5,-1)*{\cdot}, (7,-1)*{\cdot}, (9,-1)*{\cdot}, (3,1)*{\bullet}, (5,1)*{\cdot}, (5,3)*{\bullet}, (7,1)*{\bullet}, (9,1)*{\bullet}, (8,0)*{\cdot}, (11,-1)*{\bullet}, (5,-3)*{q^{8}t^2}
\end{xy}}="p2",
(4,7)*{\begin{xy}0;<.25cm,0cm>:
(0,0); (-1,-1) **@{-},(1,-1) **@{-}, (4,2); (1,-1) **@{-}, (5,1) **@{-}, (7,3); (5,1) **@{-},   (11,-1) **@{-}, (0,0)*{\circ}, (2,0)*{\circ}, (4,0)*{\cdot}, (6,0)*{\cdot}, (10,0)*{\circ}, (4,2)*{\circ}, (6,2)*{\circ}, (8,2)*{\circ}, (-1,-1)*{\bullet},  (1,-1)*{\bullet}, (3,-1)*{\cdot}, (5,-1)*{\cdot}, (7,-1)*{\cdot}, (9,-1)*{\cdot}, (3,1)*{\bullet}, (5,1)*{\bullet}, (7,1)*{\cdot}, (7,3)*{\bullet}, (9,1)*{\bullet}, (8,0)*{\cdot}, (11,-1)*{\bullet}, (5,-3)*{q^{8}t^2}
\end{xy}}="p3",
(-4,11)*{\begin{xy}0;<.25cm,0cm>:
(1,1); (-1,-1) **@{-},(2,0) **@{-}, (5,3); (2,0) **@{-}, (7,1) **@{-}, (8,2); (7,1) **@{-}, (11,-1) **@{-}, (0,0)*{\circ}, (2,0)*{\circ}, (4,0)*{\cdot}, (6,0)*{\cdot}, (10,0)*{\circ}, (4,2)*{\circ}, (6,2)*{\circ}, (8,2)*{\circ}, (-1,-1)*{\bullet},  (1,-1)*{\cdot}, (1,1)*{\bullet}, (3,-1)*{\cdot}, (5,-1)*{\cdot}, (7,-1)*{\cdot}, (9,-1)*{\cdot}, (3,1)*{\bullet}, (5,1)*{\cdot}, (5,3)*{\bullet}, (7,1)*{\bullet}, (9,1)*{\bullet}, (8,0)*{\cdot}, (11,-1)*{\bullet}, (5,-3)*{q^{9}t^3}
\end{xy}}="p4",
(0,11)*{\begin{xy}0;<.25cm,0cm>:
(1,1); (-1,-1) **@{-},(2,0) **@{-}, (2,0); (4,2) **@{-}, (5,1); (4,2) **@{-}, (6,2) **@{-}, (7,3); (6,2) **@{-}, (8,2) **@{-}, (8,2); (11,-1) **@{-}, (0,0)*{\circ}, (2,0)*{\circ}, (4,0)*{\cdot}, (6,0)*{\cdot}, (10,0)*{\circ}, (4,2)*{\circ}, (6,2)*{\circ}, (8,2)*{\circ}, (-1,-1)*{\bullet}, (1,1)*{\bullet}, (1,-1)*{\cdot}, (3,-1)*{\cdot}, (5,-1)*{\cdot}, (7,-1)*{\cdot}, (9,-1)*{\cdot}, (3,1)*{\bullet}, (5,1)*{\bullet}, (7,1)*{\cdot}, (7,3)*{\bullet}, (9,1)*{\bullet}, (8,0)*{\cdot}, (11,-1)*{\bullet}, (5,-3)*{q^{9}t^4}
\end{xy}}="p5",
(4,11)*{\begin{xy}0;<.25cm,0cm>:
(0,0); (-1,-1) **@{-}, (1,-1) **@{-}, (5,3); (1,-1) **@{-}, (6,2) **@{-}, (7,3); (6,2) **@{-}, (11,-1) **@{-}, (0,0)*{\circ}, (2,0)*{\circ}, (4,0)*{\cdot}, (6,0)*{\cdot}, (10,0)*{\circ}, (4,2)*{\circ}, (6,2)*{\circ}, (8,2)*{\circ}, (-1,-1)*{\bullet}, (1,-1)*{\bullet}, (3,-1)*{\cdot}, (5,-1)*{\cdot}, (7,-1)*{\cdot}, (9,-1)*{\cdot}, (3,1)*{\bullet}, (5,3)*{\bullet}, (5,1)*{\cdot}, (7,1)*{\cdot}, (7,3)*{\bullet}, (9,1)*{\bullet}, (8,0)*{\cdot}, (11,-1)*{\bullet}, (5,-3)*{q^9t^3}
\end{xy}}="p6",
(0,15)*{\begin{xy}0;<.25cm,0cm>:
(1,1); (-1,-1) **@{-}, (2,0) **@{-}, (5,3); (2,0) **@{-}, (6,2) **@{-}, (7,3); (6,2) **@{-}, (11,-1) **@{-}, (0,0)*{\circ}, (2,0)*{\circ}, (4,0)*{\cdot}, (6,0)*{\cdot}, (10,0)*{\circ}, (4,2)*{\circ}, (6,2)*{\circ}, (8,2)*{\circ}, (-1,-1)*{\bullet}, (1,-1)*{\cdot}, (1,1)*{\bullet}, (3,-1)*{\cdot}, (5,-1)*{\cdot}, (7,-1)*{\cdot}, (9,-1)*{\cdot}, (3,1)*{\bullet}, (5,3)*{\bullet}, (5,1)*{\cdot}, (7,1)*{\cdot}, (7,3)*{\bullet}, (9,1)*{\bullet}, (8,0)*{\cdot}, (11,-1)*{\bullet}, (5,-3)*{q^{10}t^4}
\end{xy}}="p7", "p0"; "p1" **@{-}, "p2" **@{-}, "p3" **@{-}, "p1"; "p4" **@{-}, "p5" **@{-}, "p2"; "p4" **@{-}, "p6" **@{-}, "p3"; "p5" **@{-}, "p6" **@{-}, "p7"; "p4" **@{-}, "p5" **@{-}, "p6" **@{-}
\end{xy}
\]
\caption{The induced partial order on $\B(m)$ for a Motzkin path $m \in \M(6)$.}\label{fig:thetam}
\end{figure}

\subsection{Simion and Ullman's encoding}\label{sec:SU}

In \cite{SU91}, Simion and Ullman provide a certain encoding of noncrossing partitions as words on the alphabet $\{b, e, l, r\}$. Given a noncrossing partition $\pi \in NC(n)$, the encoding assigns a word $w(\pi) = w = w_1 w_2 \cdots w_{n-1}$ of length $n-1$ as follows (see \cite[p. 199]{SU91}):

\[ w_i = \begin{cases} b & \mbox{if $i$ and $i+1$ are in different blocks }\\ & \mbox{ and $i$ is \emph{not} the largest element in its block,} \\
e & \mbox{if $i$ and $1$ are in different blocks }\\
& \mbox{ and $i+1$ is \emph{not} the smallest element in its block,}\\
l & \mbox{if $i$ and $i+1$ are in different blocks, }\\
 & \mbox{ $i$ is the largest element in its block, }\\
 & \mbox{ and $i+1$ is the smallest element in its block,}\\
r & \mbox{if $i$ and $i+1$ are in the same block.}
\end{cases}\]

We call such words \emph{SU-words}, and let \[SU(n):= \{ w(\pi) : \pi \in NC(n)\}.\] For example, if $\pi = \{ \{1,2,6\}, \{3\}, \{4,5\}\}$, we have its SU-word is $w(\pi) = rblre$. Let $B(w), E(w), L(w), R(w)$ denote the sets of positions in $w$ containing the letters $b$, $e$, $l$, and $r$, respectively. For example $w=rblre$ has $B(w) = \{2\}$, $E(w)=\{5\}$, $L(w) = \{3\}$, and $R(w) = \{1,4\}$.

By comparing the rules above with Stump's bijection $\phi$ between Dyck paths and noncrossing partitions, we get an obvious correspondence between SU-words and Dyck paths as follows. Given an SU-word $w$, we form a Dyck path by beginning with an up step, ending with a down step, and drawing the following segments for each letter of $w$:
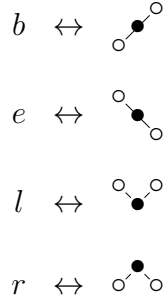
\[ 
\begin{array}{ c c c}
b & \leftrightarrow & \begin{xy}0;<.25cm,0cm>:
(0,-1)*{\circ}; (2,1)*{\circ} **@{-}, (1,0)*{\bullet}
\end{xy} \\ \\
e & \leftrightarrow & \begin{xy}0;<.25cm,0cm>:
(0,1)*{\circ}; (2,-1)*{\circ} **@{-}, (1,0)*{\bullet}
\end{xy} \\ \\
l & \leftrightarrow &
\begin{xy}0;<.25cm,0cm>:
(1,0)*{\bullet}; (0,1)*{\circ} **@{-}, (2,1)*{\circ} **@{-}
\end{xy} \\ \\
r & \leftrightarrow &
\begin{xy}0;<.25cm,0cm>:
(1,1)*{\bullet}; (0,0)*{\circ} **@{-}, (2,0)*{\circ} **@{-}
\end{xy}
\end{array}
\]

For example, with $w = rblre$, we have: 
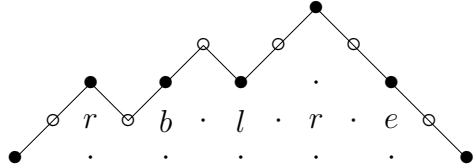
\[ \begin{xy}0;<.5cm,0cm>:
(1,1); (-1,-1) **@{-},(2,0) **@{-}, (2,0); (4,2) **@{-}, (5,1); (4,2) **@{-}, (6,2) **@{-}, (7,3); (6,2) **@{-}, (8,2) **@{-}, (8,2); (11,-1) **@{-}, (0,0)*{\circ}, (2,0)*{\circ}, (4,0)*{\cdot}, (6,0)*{\cdot}, (10,0)*{\circ}, (4,2)*{\circ}, (6,2)*{\circ}, (8,2)*{\circ}, (-1,-1)*{\bullet}, (1,1)*{\bullet}, (1,-1)*{\cdot}, (3,-1)*{\cdot}, (5,-1)*{\cdot}, (7,-1)*{\cdot}, (9,-1)*{\cdot}, (3,1)*{\bullet}, (5,1)*{\bullet}, (7,1)*{\cdot}, (7,3)*{\bullet}, (9,1)*{\bullet}, (8,0)*{\cdot}, (11,-1)*{\bullet}, (1,0)*{r}, (3,0)*{b}, (5,0)*{l}, (7,0)*{r}, (9,0)*{e}
\end{xy}\]
By abuse of notation, let $w(p) \in SU(n)$ denote the SU-word associated to the Dyck path $p$, i.e., $w(p) = w(\phi(p))$. 

We have the following.

\begin{obs}[\cite{SU91}, Observation 2.3]\label{obs:rk}
Let $p \in \Dy(n)$, and let $w = w(p)$. Then, \begin{align*}
 n&=|B(w)|+|E(w)|+|L(w)|+|R(w)|+1,\\
  &=2|B(w)| + |L(w)|+|R(w)|+1,
 \end{align*} and \[ \rk(p) = |B(w)|+|R(w)|.\]
\end{obs}

We can also map an SU-word to a Motzkin path via the correspondence with Dyck paths, i.e., if $w = w(p)$, we declare $\theta(w) = \theta(p)$. Explicitly, this becomes the path constructed from the word by:
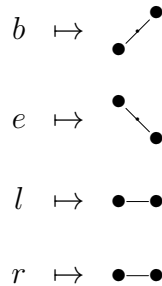
\[
\begin{array}{ c c c}
b & \mapsto & \begin{xy}0;<.25cm,0cm>:
(0,-1)*{\bullet}; (2,1)*{\bullet} **@{-}, (1,0)*{\cdot}
\end{xy} \\ \\
e & \mapsto & \begin{xy}0;<.25cm,0cm>:
(0,1)*{\bullet}; (2,-1)*{\bullet} **@{-}, (1,0)*{\cdot}
\end{xy} \\ \\
l & \mapsto &
\begin{xy}0;<.25cm,0cm>:
(0,0)*{\bullet}; (2,0)*{\bullet} **@{-}
\end{xy} \\ \\
r & \mapsto &
\begin{xy}0;<.25cm,0cm>:
(0,0)*{\bullet}; (2,0)*{\bullet} **@{-}
\end{xy}
\end{array}
\]

Under this correspondence, we see the following.

\begin{obs}\label{obs:classes}
Let $p, p \in \Dy(n)$. Then $\theta(p) = \theta(p')$ if and only if  $B(w(p)) = B(w(p'))$ and $E(w(p)) = E(w(p'))$. 
\end{obs}

That is, the only places in which $w(p)$ and $w(p')$ differ are that some letters $l$ and $r$ are exchanged. As previously observed, this means that $p$ and $p'$ differ only in that some up-down steps between odd vertices are replaced by down-up steps.

\subsection{Refined Catalan recurrence and continued fraction expansion} 

We will now show the polynomials $\Dy(n;q,t)$ satisfy the refined Catalan recurrence of Theorem \ref{thm:recurrence}. We wish to show \[ \Dy(n;q,t) = \sum_{k=0}^{n-1} (qt)^k \Dy(k; q, 1/t)\Dy(n-1-k; q, t).\]

The proof is based on the usual combinatorial identity  for Dyck paths: \[ \Dy(n) \cong \bigcup_{k=0}^{n-1} \Dy(k)\times \Dy(n-1-k),\] whereby we decompose a Dyck path according to its point of first return to the $x$-axis. The path to the left of this point is associated to a Dyck path after removing its first up step and its last down step. The path to the right is untouched. See Figure \ref{fig:decomp}.

\begin{figure}
\[
\begin{array}{c}
\begin{xy}0;<.5cm,0cm>:
(-1,2)*{p=}, (3,3); (0,0) **@{-}, (4,2) **@{-}, (6,4); (4,2) **@{-}, (10,0) **@{-}, (12,2); (10,0) **@{-}, (14,0) **@{-}, (15,1); (14,0) **@{-}, (16,0) **@{-}, (10,-1); (10,5) **@{--}, (10,1); (0,1) **@{--},(0,0)*{\bullet}, (1,1)*{\bullet}, (2,2)*{\bullet}, (3,3)*{\bullet},  (6,4)*{\bullet}, (7,3)*{\bullet},  (12,2)*{\bullet}, (15,1)*{\bullet}, (16,0)*{\bullet}, (2,0)*{\bullet}, (3,1)*{\bullet}, (4,2)*{\bullet}, (5,3)*{\bullet}, (4,0)*{\bullet}, (5,1)*{\bullet}, (6,2)*{\bullet}, (6,0)*{\bullet}, (7,1)*{\bullet}, (8,2)*{\bullet},  (8,0)*{\bullet}, (9,1)*{\bullet},   (10,0)*{\bullet}, (11,1)*{\bullet}, (12,0)*{\bullet}, (13,1)*{\bullet}, (14,0)*{\bullet}
\end{xy}\\ \\
\updownarrow \\ \\
(p_1,p_2) = \left( 
\begin{array}{c}\begin{xy}0;<.5cm,0cm>:
(2,2); (0,0) **@{-}, (3,1) **@{-}, (5,3); (3,1) **@{-}, (8,0) **@{-}, (0,0)*{\bullet}, (1,1)*{\bullet}, (2,2)*{\bullet}, (2,0)*{\bullet}, (3,1)*{\bullet}, (4,2)*{\bullet}, (5,3)*{\bullet}, (4,0)*{\bullet}, (5,1)*{\bullet}, (6,2)*{\bullet}, (6,0)*{\bullet}, (7,1)*{\bullet},   (8,0)*{\bullet}, (9,0)*{,}
\end{xy}
\end{array}
\begin{array}{c}\begin{xy}0;<.5cm,0cm>:
(2,2); (0,0) **@{-}, (4,0) **@{-}, (5,1); (4,0) **@{-}, (6,0) **@{-}, (0,0)*{\bullet}, (1,1)*{\bullet}, (2,2)*{\bullet}, (2,0)*{\bullet}, (3,1)*{\bullet}, (4,0)*{\bullet}, (5,1)*{\bullet}, (6,0)*{\bullet}
\end{xy}
\end{array}
\right)
\end{array}
\]
\caption{Joining or decomposing Dyck paths according to the first return.}\label{fig:decomp}
\end{figure}
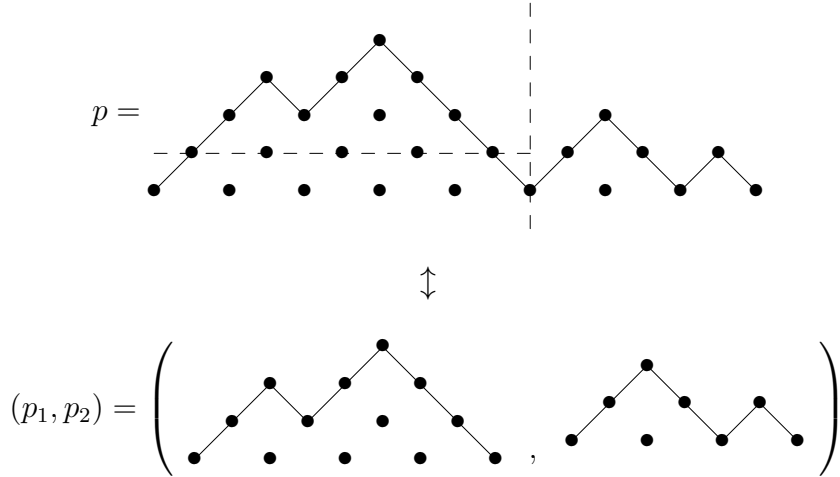

Fix a path $p \in \Dy(n)$ whose first return to the $x$-axis is after $2(k+1)$ steps. We let $p_1\in \Dy(k)$ be the path given by taking the part of this initial segment of $p$ and removing the first and last steps. We let  $p_2\in \Dy(n-1-k)$ be the part of $p$ beginning at the point of first return.

This process is clearly reversible. Let $p^+$ denote the path obtained by adding an up step at the beginning of $p$ and a downstep at the end of $p$. Then if we are given a pair of paths $p_1 \in \Dy(k)$ and $p_2 \in \Dy(n-1-k)$, we create a new path by identifying the last node of $p_1^+$ with the first node of $p_2$. The new path has a total of $2k+2 + 2(n-1-k) = 2n$ steps, and the point of first return in the new path occurs after $2(k+1)$ steps.

The decomposition by first return clearly respects both area and rank in that \[ \area(p) = \area(p_1^+) + \area(p_2) \quad \mbox{ and } \quad  \rk(p) = \rk(p_1^+) + \rk(p_2).\] We can be more precise, however, as shown in the following result. 

\begin{prp}\label{prp:decomp}
Suppose a Dyck path $p \in \Dy(n)$ has its first return after $2(k+1)$ steps, and hence decomposes as $p = p_1^+ p_2$, with $p_1 \in \Dy(k)$ and $p_2 \in \Dy(n-1-k)$.  Then \begin{align*}
 \area(p_1^+) &= k+\area(p_1), \\
  \rk(p_1^+) &= k-\rk(p_1),
  \end{align*} and hence
\begin{align*}
\area(p) &= k+\area(p_1) + \area(p_2),\\
\rk(p) &= k-\rk(p_1) + \rk(p_2).
\end{align*}
\end{prp}

\begin{proof}
It is quite easy to see that $\area(p_1^+) = k+\area(p_1)$; apart from $(0,0)$ and $(2k+2,0)$, there are precisely $k$ lattice points in $p_1^+$ that lie below the line $y=1$.

We will now prove that $\rk(p^+_1) = k-\rk(p_1)$ by considering SU-words. 

Recall from Observation \ref{obs:rk} that the rank of a Dyck path is the number of letters $b$ and $r$ in its SU-word. Now consider how ``raising" a path affects SU-words. For instance, suppose $l$ and $b$ are consecutive letters in an SU-word. Then in the corresponding part of the Dyck path, we see: 
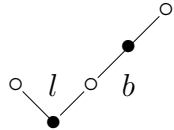
\[ \begin{xy}0;<.5cm,0cm>:
(0,-1)*{\circ};  (1,0)*{\bullet}, (-1,-2)*{\bullet} **@{-}, (2,1)*{\circ} **@{-}, (-2,-1)*{\circ}; (-1,-2) **@{-}, (-1,-1)*{l}, (1,-1)*{b}, 
\end{xy}\] However, in adding a new up step at the beginning of the path and a new down step at the end, the odd and even nodes get swapped (white for black), and hence we see the gap between an $l$ and a $b$ in a path $p$ would correspond to the letter $b$ (between the black nodes above) in the path $p^+$. 

By going through all pairs of letters  we can see how the raising operator behaves with SU-words. In the array below, we mean that if $cd$ are consecutive letters in $w(p)$, we replace them in $w(p^+)$ with the letter in the row indexed by $c$, column indexed by $d$.
\[ \begin{array}{ c | c c c c c}
 & b & e & l & r & \cdot\\
 \hline
 b & b & r & r & b & \\
 e & l & e & e & l & e\\
 l & b & r & r & b & r\\
 r & l & e & e & l & e\\
 \cdot & b &  & r & b & r
 \end{array}\]  We use a $\cdot$ on the far left and far right to indicate the transitions at the beginning and end of the word. Notice that a word cannot begin with an $e$ nor can it end with a $b$. So, for example, if $w(p) = rblre$ we look at the gaps between the letters and use the array to find: \[ \begin{array}{ c c c c c c c c c c c c c}
 \cdot & & r & & b & & l & & r & & e & & \cdot \\
  & \downarrow & & \downarrow & & \downarrow & & \downarrow & & \downarrow & & \downarrow\\
 & b & & l & & r & & b & & e & & e
 \end{array}\] and so $w(p^+) = blrbee$.

Now, as in the statement of the proposition, suppose $p_1 \in \Dy(k)$, and let $w=w(p_1) = w_1\cdots w_{k-1}$, $w'= w(p_1^+) = w_1' \cdots w'_k$. Then by looking at the transformation array above, we see that $w'_1 \in \{b,r\}$ and for $1\leq i \leq k-1$, $w_i \in \{b, l\}$ if and only if $w'_{i+1} \in \{b,r\}$. Hence,
\begin{align*} 
\rk(p_1^+) &= |B(w')|+|R(w')|, \\
       &= |B(w)|+|L(w)|+1, \\
       &= k-|B(w)|-|R(w)|, \\ 
	&= k- \rk(p_1),
\end{align*}
where the final two equalities are direct applications of Observation \ref{obs:rk}.

The proposition now follows.
\end{proof}

Let us now prove Theorem \ref{thm:recurrence}.

\begin{proof}[Proof of Theorem \ref{thm:recurrence}]
In light of Proposition \ref{prp:decomp}, we express the polynomial as follows.
\begin{align*}
\Dy(n;q,t) &= \sum_{p \in \Dy(n)} q^{\area(p)}t^{\rk(p)}\\
 &=\sum_{k=0}^{n-1} \sum_{(p_1,p_2) \in \Dy(k)\times \Dy(n-1-k)} q^{\area(p_1^+)+\area(p_2)}t^{\rk(p_1^+)+\rk(p_2)} \\
 &=\sum_{k=0}^{n-1} \left(\sum_{p_1 \in \Dy(k)} q^{\area(p_1^+)}t^{\rk(p_1^+)} \right)\cdot\left(\sum_{p_2 \in \Dy(n-1-k)} q^{\area(p_2)}t^{\rk(p_2)}\right)\\
 &=\sum_{k=0}^{n-1} \left(\sum_{p_1 \in \Dy(k)} q^{k+\area(p_1)} t^{k-\rk(p_1)}\right)\cdot\Dy(n-1-k;q,t) \\
 &=\sum_{k=0}^{n-1} (qt)^k\cdot\left(\sum_{p_1 \in \Dy(k)} q^{\area(p_1)} t^{-\rk(p_1)}\right)\cdot\Dy(n-1-k;q,t)\\
 &=\sum_{k=0}^{n-1} (qt)^k\Dy(k;q,1/t)\Dy(n-1-k;q,t),
\end{align*}
as desired.
\end{proof}

Carlitz \cite{Car74} obtained a continued fraction expression for the ordinary generating function 
\[
\Dy(q,z)=\sum_{n\geq 0}\sum_{p\in\Dy(n)}q^{\area(p)}z^{n}.
\]
To be precise, he showed 
\begin{equation}\label{eq:continuedfraction}
\Dy(q,qz) = \sum_{n\geq 0}\sum_{p\in\Dy(n)}q^{\area'(p)}z^{n}=\displaystyle\frac{1}{1-\displaystyle\frac{qz}{1-\displaystyle\frac{q^{2}z}{1-\displaystyle\frac{q^{3}z}{1-\displaystyle\frac{q^{4}z}{\ddots}}}}},
\end{equation}
where $\area'(p):=\area(p)+n$ if $p\in\Dy(n)$ (see also~\cite{F80}).

\begin{rmk}
The $\area'$ statistic counts the number of $1/\sqrt{2}\times1/\sqrt{2}$ diamonds that can fit below the path if they are allowed to hang halfway below the $x$-axis. So a path $p\in \Dy(n)$ with $\area'(p)=k$ corresponds to a \emph{fountain of coins} with $n$ coins in the bottom row and $k$ coins in total (see ~\cite{BCS02,OW88}). The $\area'$ statistic, as well as the fountain of coins problem, provided combinatorial interpretations for the \emph{Rogers-Ramanujan} continued fraction (see~\cite{A79, A98}).
\end{rmk}

We obtain an analogous result for the generating function that tracks the rank and area of elements of $\Dy(n)$. Define the function \[\Dy(q,t,z)=\sum_{n\geq 0}\Dy(n;q,t)z^n = \sum_{n\geq 0}\sum_{p \in \Dy(n)} q^{\area(p)}t^{\rk(p)}z^n.\]

\begin{prp}\label{prp:cf}
The ordinary generating function $\Dy(q,t,z)$ has the following continued fraction expansion:
\begin{equation*}
\Dy(q,t,z) =\frac{1}{1-\displaystyle\frac{z}{1-\displaystyle\frac{qtz}{1-\displaystyle\frac{q^2z}{1-\displaystyle\frac{q^3tz}{1-\displaystyle\frac{q^4z}{1-\displaystyle\frac{q^5tz}{\ddots}}}}}}},
\end{equation*}
i.e., for $k\geq 1$, the numerator of level $2k$ is $q^{2k-1}tz$ and the numerator of level $2k-1$ is $q^{2k-2}z$.
\end{prp}

\begin{proof}
From Theorem \ref{thm:recurrence}, we get, for $n\geq 1$, \[\Dy(n;q,t) = \sum_{k=0}^{n-1} (qt)^k \Dy(k; q, 1/t)\Dy(n-1-k; q, t).\]

Hence, 
\begin{align*}
 \Dy(q,t,z) &= \sum_{n\geq 0} \Dy(n;q,t) z^n \\
  &= 1 + \sum_{n\geq 1} \left(\sum_{k=0}^{n-1} (qt)^k \Dy(k; q, 1/t)\Dy(n-1-k; q, t)\right) z^n \\
  &= 1 + z\cdot\sum_{n\geq 1} \sum_{k=0}^{n-1} \left( \Dy(k; q, 1/t)(qtz)^k\right)\left(\Dy(n-1-k; q, t)z^{n-1-k}\right) \\
  &= 1 + z \Dy(q,1/t,qtz)\Dy(q,t,z).
\end{align*}

From here, we have \begin{equation}\label{eq:cf1} \Dy(q,t,z) = \frac{1}{1-z\Dy(q,1/t,qtz)}.\end{equation}
But also, \begin{equation}\label{eq:cf2} \Dy(q,1/t,qtz) = \frac{1}{1-qtz\Dy(q,t,q^2z)},\end{equation} and by putting \eqref{eq:cf2} into \eqref{eq:cf1}, we have \[ \Dy(q,t,z) = \frac{1}{1-\displaystyle\frac{z}{1-qtz\Dy(q,t,q^2z)}}.\]
Continuing in this way, we find:
\begin{align*}
\Dy(q,t,z)&=\frac{1}{1-\displaystyle\frac{z}{1-qtz\Dy(q,t,q^2z)}}\\
&=\frac{1}{1-\displaystyle\frac{z}{1-\displaystyle\frac{qtz}{1-q^2z\Dy(q,1/t,q^3tz)}}}
\\
&=\frac{1}{1-\displaystyle\frac{z}{1-\displaystyle\frac{qtz}{1-\displaystyle\frac{q^2z}{1-q^3tz\Dy(q,t,q^4z)}}}}\\
&\quad\vdots\\
&=\frac{1}{1-\displaystyle\frac{z}{1-\displaystyle\frac{qtz}{1-\displaystyle\frac{q^2z}{1-\displaystyle\frac{q^3tz}{1-\displaystyle\frac{q^4z}{1-\displaystyle\frac{q^5tz}{\ddots}}}}}}},
\end{align*}
as desired.
\end{proof}

If we set $q=1$, we get
\begin{align*} 
\Dy(1,t,z) &= \frac{1}{1-\displaystyle\frac{z}{1-\displaystyle\frac{tz}{1-\displaystyle\frac{z}{1-\displaystyle\frac{tz}{1-\displaystyle\frac{z}{1-\displaystyle\frac{tz}{\ddots}}}}}}}\\
&=\frac{1}{1-\displaystyle\frac{z}{1-tz\Dy(1,t,z)}},
\end{align*}
which in turn gives
\[tz(\Dy(1,t,z))^2 -(1+z(t-1))\Dy(1,t,z)  +1=0.\]
Solving for $\Dy(1,t,z)$ gives the generating function for the Narayana numbers:
\[ \Dy(1,t,z) = \frac{1+z(t-1)-\sqrt{1-2z(t+1)+z^2(t-1)^2}}{2tz}.\]

\subsection{Simion and Ullman's symmetric boolean decomposition}

Suppose $(P,\leq)$ is a graded poset of rank $n$, and let $B_j$ denote the boolean algebra on $j$ elements. 

\begin{defn}[Symmetric boolean decomposition]\label{def:sbd}
A \emph{symmetric boolean decomposition} of $(P,\leq)$ is a partition of the elements of $P$, $\{ P_1,\ldots,P_k\}$, such that for each $i = 1,\ldots,k$:
\begin{itemize}
\item there is a $j$, $0\leq j\leq n/2$, such that $|P_i|=2^{n-2j}$,
\item there is a unique minimum element of $P_i$ and its rank is $j$,
\item the boolean algebra $B_{n-2j}$ is isomorphic to a subposet of the induced poset $(P_i, \leq)$ in $(P,\leq)$. 
\end{itemize}
\end{defn}

Loosely speaking, a symmetric boolean decomposition is a partition of a poset into disjoint boolean intervals whose middle ranks coincide with the middle rank of the entire poset.

Simion and Ullman used their SU-word encoding to give a constructive proof that the lattice $(NC(n),\leq)$ has a symmetric chain decomposition. In fact, as they themselves observed, they obtained something stronger: a symmetric boolean decomposition. We summarize their result with the following Theorem. (See Sections 2 and 3 of \cite{SU91}. See also \cite{PetersenShards} for a different symmetric boolean decomposition.)

\begin{thm}[\cite{SU91}]\label{thm:SBD}
For any $n\geq 1$, we have the following partition of the noncrossing partition lattice. In terms of SU-words, \[ SU(n) = \bigcup_{m \in \M(n)} \{ w \in SU(n) : \theta(w) = m\}.\] In terms of Dyck paths, \[ \Dy(n) = \bigcup_{m \in \M(n)} \B(m).\] Moreover, if $\rk(m) = i$, then as an induced subposet of $(\Dy(n),\leq)$, \[ (\B(m),\leq) \cong B_{n-1-2i}.\]
\end{thm}

In particular, $(\B(m),\leq)$ is isomorphic to the boolean algebra on horizontal steps of $m$: each horizontal step can be chosen independently to be either down-up (letter $l$) or up-down (letter $r$). See Observation \ref{obs:classes}. We see an example of $\B(m)$ in Figure \ref{fig:thetam}, and in Figure \ref{fig:NC4} we see the entire symmetric boolean decomposition of $(\Dy(4),\leq)$ highlighted with bold lines.

We now consider the distribution of area and rank on $\B(m)$. 

\begin{prp}\label{prp:gam1}
Let $m \in \M(n)$ be a Motzkin path. Then, \[ \sum_{p \in \B(m)} q^{\area(p)}t^{\rk(p)} = q^{\area(m)}t^{\rk(m)}(1+qt)^{n-1-2\rk(m)}.\]
\end{prp}

\begin{proof}
Given a Motzkin path $m \in \M(n)$, we have \[ \area(m) = \area(\rho(m)) \quad \mbox{ and } \quad \rk(m) = \rk(\rho(m)) = |B(w(\rho(m)))|.\] Moreover, each of the horizontal steps of $m$ becomes a down-up step between odd vertices of $\rho(m)$. In terms of SU-words, we can get any element of $\B(m)$ from $\rho(m)$ by switching a letter $l$ for a letter $r$. By Observation \ref{obs:rk}, the rank increases by one for each letter $r$. In terms of Dyck paths, swapping an $l$ for an $r$ amounts to: 
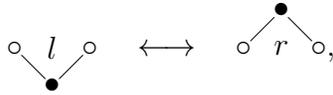
\[ 
\begin{xy}0;<.5cm,0cm>:
(1,-1)*{\bullet}; (0,0)*{\circ} **@{-}, (2,0)*{\circ} **@{-}, (1,0)*{l}
\end{xy} \quad \longleftrightarrow \quad 
\begin{xy}0;<.5cm,0cm>:
(1,1)*{\bullet}; (0,0)*{\circ} **@{-}, (2,0)*{\circ} **@{-}, (1,0)*{r}
\end{xy},
\] and this operation on Dyck paths clearly increases area by one.

Hence, for each swap $l \to r$, we increase both area and rank by one. The result follows.
\end{proof}

Together, Theorem \ref{thm:SBD} and Proposition \ref{prp:gam1} establish the refined $\gamma$-nonnegativity for $\Dy(n;q,t)$ claimed in Theorem \ref{thm:main}.

\section{Further questions}\label{sec:further}

We now discuss some related questions.

\subsection{The Bandlow-Killpatrick bijection}\label{sec:BKS}

Stump's bijection $\phi$ between Dyck paths and permutations is not the first to take the area of a path to the length of a permutation. In \cite{BK01}, Bandlow and Killpatrick give the following bijection between Dyck paths and 312-avoiding permutations, i.e., permutations for which there is a triple $i < j < k$ with $\sigma(j) < \sigma(k) < \sigma(i)$.

We fill in the boxes under the Dyck path with simple transpositions ($s_i = (i, i+1)$) as follows:
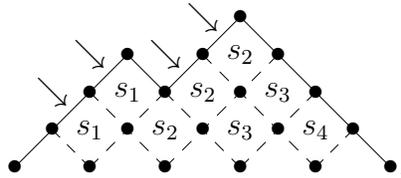
\[ \begin{xy}0;<.5cm,0cm>:
(3,3); (0,0) **@{-}, (4,2) **@{-}, (6,4); (4,2) **@{-}, (10,0) **@{-}, (2,0); (1,1) **@{--}, (4,2) **@{--}, (4,0); (2,2) **@{--}, (7,3) **@{--}, (6,0); (4,2) **@{--}, (8,2) **@{--}, (8,0); (5,3) **@{--}, (9,1) **@{--}, (0,0)*{\bullet}, (1,1)*{\bullet}, (2,2)*{\bullet}, (3,3)*{\bullet},  (6,4)*{\bullet}, (2,0)*{\bullet}, (3,1)*{\bullet}, (4,2)*{\bullet}, (5,3)*{\bullet}, (4,0)*{\bullet}, (5,1)*{\bullet}, (6,2)*{\bullet}, (6,0)*{\bullet}, (7,1)*{\bullet}, (8,2)*{\bullet}, (8,0)*{\bullet}, (9,1)*{\bullet}, (10,0)*{\bullet}, (7,3)*{\bullet}, (2,1)*{s_1}, (3,2)*{s_1}, (4,1)*{s_2}, (5,2)*{s_2}, (6,3)*{s_2}, (6,1)*{s_3}, (7,2)*{s_3}, (8,1)*{s_4}, (5,4)*{\searrow}, (4,3)*{\searrow}, (2,3)*{\searrow}, (1,2)*{\searrow}
\end{xy},
\]
then we form a word by reading the entries in the boxes down the diagonals indicated from right to left, e.g., \[ s_2 s_3 s_4 \,\, s_2 s_3 \,\, s_1 s_2 \,\, s_1 = 35421.\] 

Bandlow and Killpatrick prove that such a word is reduced, i.e., that length is equal to area, and moreover that the permutation thus produced avoids the pattern 312. 

Similarly, Stump \cite[Section 2.2]{Stump} gives a bijection with 231-avoiding permutations by the following method of filling boxes and reading generators:
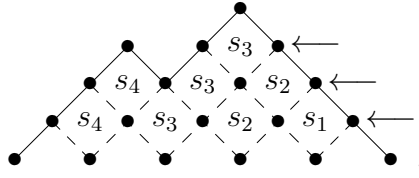
\[ \begin{xy}0;<.5cm,0cm>:
(3,3); (0,0) **@{-}, (4,2) **@{-}, (6,4); (4,2) **@{-}, (10,0) **@{-}, (2,0); (1,1) **@{--}, (4,2) **@{--}, (4,0); (2,2) **@{--}, (7,3) **@{--}, (6,0); (4,2) **@{--}, (8,2) **@{--}, (8,0); (5,3) **@{--}, (9,1) **@{--}, (0,0)*{\bullet}, (1,1)*{\bullet}, (2,2)*{\bullet}, (3,3)*{\bullet},  (6,4)*{\bullet}, (2,0)*{\bullet}, (3,1)*{\bullet}, (4,2)*{\bullet}, (5,3)*{\bullet}, (4,0)*{\bullet}, (5,1)*{\bullet}, (6,2)*{\bullet}, (6,0)*{\bullet}, (7,1)*{\bullet}, (8,2)*{\bullet}, (8,0)*{\bullet}, (9,1)*{\bullet}, (10,0)*{\bullet}, (7,3)*{\bullet}, (2,1)*{s_4}, (3,2)*{s_4}, (4,1)*{s_3}, (5,2)*{s_3}, (6,3)*{s_3}, (6,1)*{s_2}, (7,2)*{s_2}, (8,1)*{s_1}, (8,3)*{\longleftarrow}, (9,2)*{\longleftarrow}, (10,1)*{\longleftarrow}
\end{xy},
\] which gives \[s_3 \,\, s_2s_3s_4 \,\, s_1s_2s_3s_4 = 54213. \]

It would be very interesting to find a statistic $s$ for 312-avoiding permutations such that $\rk(p)$ equals $s(\sigma)$. Let $S_n(312)$ denote the set of 312-avoiding (resp. 231-avoiding) permutations of $n$. 

\begin{ques} 
Can we find a statistic $s$ such that \[ \sum_{p \in \Dy(n)} q^{\area(p)} t^{\rk(p)} = \sum_{\sigma \in S_n(312)} q^{\ell(\sigma)} t^{s(\sigma)} \quad \left(=\sum_{\sigma \in S_n(231)} q^{\ell(\sigma)} t^{s(\sigma)}\right)?\]
\end{ques}

We remark that it is known that \[\sum_{\sigma \in S_n(312)} t^{d(\sigma)} =  \sum_{\sigma \in S_n(231)} t^{d(\sigma)} = \sum_{p \in \Dy(n)} t^{\rk(p)},\] where $d(\sigma) = |\{ i : \sigma(i) > \sigma(j) \}|$ is the number of \emph{descents} of $\sigma$. However, $s\neq d$, since, for example, the long element also has the greatest number of descents. Other guesses for ``Eulerian" statistics, like excedance number, ascents, inverse descents, and inverse ascents are quickly ruled out as well. We should note, however, that since the sum is only over 312-avoiding (resp. 231-avoiding) permutations there is no reason to expect $s(\sigma)$ to give the Eulerian distribution over all of $S_n$.

\subsection{Inversions and excedances in the symmetric group}\label{sec:excinv}

It is straightforward to prove that for $\sigma \in [e, (12\cdots n)]$, reflection length equals \emph{excedance number}, i.e., we have $\ell'(\sigma) =\exc(\sigma) := |\{ i : \sigma(i) > i\}|$.

As we can see from Theorem \ref{thm:main} and Corollary \ref{cor:ll'}, the distribution of inversions and excedances for permutations in the interval $[e,(12\cdots n)]$ exhibits $\mathbb{Z}_{\geq 0}[q]$ $\gamma$-nonnegativity, i.e.,
\[ \sum_{\sigma \in [e,(12\cdots n)]} q^{\inv(\sigma)-\exc(\sigma)}t^{\exc(\sigma)} = \sum_{0\leq j\leq (n-1)/2} \gamma_j(q) t^j(1+t)^{n-1-2j}.\]
It appears that the same sort of result holds when taking the sum over the entire symmetric group. Define \[ S_n(q,t) = \sum_{ \sigma \in S_n} q^{\inv(\sigma)} t^{\exc(\sigma)}.\]

As examples, we have:
\begin{align*}
 S_4(q,t) &= (1+qt)^3 + (2q^2+3q^3 + 2q^4+q^5)t(1+qt)\\
 S_5(q,t) &= (1+qt)^4 + (3q^2+ 5q^3 + 5q^4 + 3q^5 +q^6)t(1+qt)^2 \\
  &\qquad + (q^4 + 2q^5 + 2q^6+ 3q^7 + 4q^8 + 3q^9 + q^{10} )t^2
\end{align*}

\begin{conj}
There exist polynomials $\gamma_j(q)$ with nonnegative integer coefficients such that: \[ S_n(q,t/q)=\sum_{\sigma \in S_n} q^{\inv(\sigma)-\exc(\sigma)}t^{\exc(\sigma)} = \sum_{0\leq j \leq (n-1)/2} \gamma_j(q) t^j(1+t)^{n-1-2j}.\]
\end{conj} 

We remark that Shareshian and Wachs \cite{SW10} show that $\sum_{\sigma \in S_n} q^{\maj(\sigma)}t^{\exc(\sigma)}$ has an expansion of this type.

In work of Clarke, Steingr\'imsson, and Zeng \cite[Corollary 12]{CSZ97} it is shown that the polynomials $S_n(q,t/q)$ are symmetric in $t$, but of course $\gamma$-nonnegativity is stronger. (Note that their Table I has an error, in that the entry of row $k=3$, column $m=14$ should be a 4, not a 1.) We note that when $q=1$, we get the Eulerian polynomials, which are well-known to be $\gamma$-nonnegative with integer coefficients, a result first due to Foata and Sch\"utzenberger \cite{FS70}. 

\subsection{Noncrossing partitions of type B}\label{sec:B}

We now show that the noncrossing partitions of type B admit a symmetric boolean decomposition, though we have no companion to Theorem \ref{thm:main}.

\begin{defn}
A \emph{$B_{n}$-partition} is a partition $\pi$ of the set $[\pm n]:=\{\pm1,\pm2,\ldots,\pm n\}$ (also called a \emph{partition of type B}) into blocks $B_{1},B_{2},\ldots, B_{m}$ satisfying the following properties:
\begin{enumerate}
\item If $B$ is a block of $\pi$, then $-B$ is also a block of $\pi$
\item There is at most one block of the form $\{i,-i\}$, $i\in[n]$. This block is called the \emph{zero block} of $\pi$.
\end{enumerate}
\end{defn}	

One can represent $B_{n}$-partitions by placing the elements of $[\pm n]$ around a circle using the order $1,2,\ldots,n,-1,-2,\ldots,-n$ and then drawing chords between two numbers $a$ and $b$ if they belong to the same block and there are no elements between $a$ and $b$ (clockwise) that belong to the same block. A $B_{n}$-partition, or a noncrossing partition of type B, is said to be \emph{noncrossing} if the chords do not cross. Figure~\ref{fig:Bn-noncrossing} depicts two $B_{3}$-noncrossing partitions.

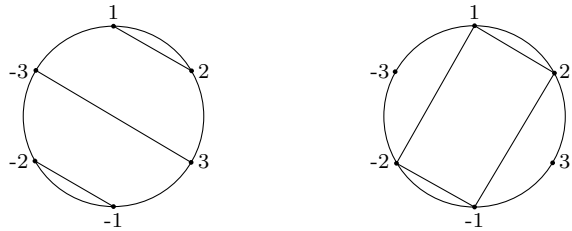
\begin{figure}[h!]
\begin{tikzpicture}[scale=.6]
\draw(1,3) circle (2cm);
\draw (1,5)-- (2.73,4.01);
\draw (-0.72,4.02)-- (2.72,1.98);
\draw (1,1)-- (-0.74,2.01);
\draw(9,5.32) node {\tiny 1};
\draw(1,5.32) node {\tiny 1};
\draw(9,.7) node {\tiny -1};
\draw(1,.7) node {\tiny -1};
\draw(6.9,4) node {\tiny -3};
\draw(-1.1,4) node {\tiny -3};
\draw(11,4) node {\tiny 2};
\draw(3,4) node {\tiny 2};
\draw(6.9,2) node {\tiny -2};
\draw(-1.1,2) node {\tiny -2};
\draw(11,2) node {\tiny 3};
\draw(3,2) node {\tiny 3};
\draw(9,3) circle (2.01cm);
\draw (7.27,1.96)-- (9,0.99);
\draw (9,0.99)-- (10.77,3.96);
\draw (10.77,3.96)-- (9,5.01);
\draw (9,5.01)-- (7.27,1.96);
\fill (1,5) circle (1.5pt);
\fill (2.73,4.01) circle (1.5pt);
\fill (2.72,1.98) circle (1.5pt);
\fill (-0.74,2.01) circle (1.5pt);
\fill (-0.72,4.02) circle (1.5pt);
\fill (9,5.01) circle (1.5pt);
\fill (10.77,3.96) circle (1.5pt);
\fill (10.73,1.97) circle (1.5pt);
\fill (9,0.99) circle (1.5pt);
\fill (7.27,1.96) circle (1.5pt);
\fill (7.24,3.99) circle (1.5pt);
\fill (1,1) circle (1.5pt);
\end{tikzpicture}
\caption{The $B_{3}$-noncrossing partitions $\{\{1,2\},\{-1,-2\}\{-3,3\}\}$ and $\{\{-1,-2,1,2\},\{3\},\{-3\}\}$.}
\label{fig:Bn-noncrossing}
\end{figure}

The poset of $B_{n}$-noncrossing partitions ordered by reverse refinement is a ranked lattice with the rank of $\pi$ being $n-\nzb(\pi)$. One says that $\nzb(\pi)=k$ if $\pi$ has $2k$ nonzero blocks (see~\cite[Proposition 2]{R97}). This lattice is denoted by $(NC_{B}(n),\leq)$.

In \cite[Proposition 6]{R97}, Reiner gives a bijection between noncrossing partitions of type B and what we will call $(L,R)$ pairs. That is, \[ NC_B(n) \leftrightarrow \{ (L,R) \in [n]\times [n] : |L|=|R|\}=P_B(n).\] 
We will describe the map $\eta: P_B(n) \to NC_B(n)$ of \cite[Proposition 6]{R97} with an example. Suppose $n=5$ and let $(L,R) = (\{2,3,4\}, \{1,4,5\})$. We write the string \[ 1 \quad 2 \quad 3 \quad 4 \quad 5 \quad -1 \quad -2 \quad -3 \quad -4 \quad -5,\] and for each $i$ in $L$ we put a left parenthesis just to the left of both $i$ and $-i$: \[ 1 \quad (2 \quad (3 \quad (4 \quad 5 \quad -1 \quad (-2 \quad (-3 \quad (-4 \quad -5.\] Next we place a right parenthesis just to the right of $j$ and $-j$ for every $j$ in $R$:  \[ 1) \quad (2 \quad (3 \quad (4) \quad 5) \quad -1) \quad (-2 \quad (-3 \quad (-4) \quad -5).\] The parenthesization (read cyclically) gives: \[ \{ \{ 1, -2\}, \{-1,2\}, \{3, 5\}, \{-3,-5\}, \{4\} , \{-4\}\} = \eta(L,R). \]

Under the bijection $\eta$, rank corresponds to $n-|L|$, and it is clear that if $\pi < \pi'$, then $L(\pi') \subset L(\pi)$ and $R(\pi') \subset R(\pi)$. However, the converse is not true, i.e., simply because $L' \subset L$ and $R' \subset R$ it does not necessarily follow that the corresponding $B_{n}$-partitions are comparable in $NC_B(n)$. For example we have \[ \eta( \{3,4\} , \{1,5\} ) = \{ \{ 1, -3\}, \{-1,3\}, \{2,-2\}, \{4,5\}, \{-4,-5\}\},\] and when comparing with our first example, we see \[ \eta (\{2,3,4\}, \{1,4,5\}) \nleq \eta( \{3,4\}, \{1,5\}).\]

Nonetheless, something weaker is true, which will allow us to give an explicit symmetric boolean decomposition of $NC_B(n)$. (The first proof of this fact is due to Hersh \cite{H99}. An inductive proof was given in \cite{PetersenShards}.) Notice that if the same element appears in both $L$ and $R$, then that element is a singleton block in $\eta(L,R)$. Hence, removing common elements from $(L,R)$ pairs will always coarsen the corresponding partition. We have the following.

\begin{lem}\label{lem:B}
Suppose $(L,R) \in [n]\times [n]$, and that there is an element $i \in L\cap R$. Then \[ \eta(L,R) \leq \eta(L-\{i\}, R-\{i\}).\] 
\end{lem}

Define the set of pairs with no elements in common: \[ P_{B,\emptyset}(n) = \{ (L,R) \in [n]\times [n] : |L|=|R|, L\cap R = \emptyset\}.\] Further, for any pair $(L,R) \in P_{B,\emptyset}(n)$, let \[ \B(L,R) = \{ (L',R') : L' = L \cup A, R'= R \cup A \mbox{ for some } A \in [n]-L-R\}.\] 

If we partially order $\B(L,R)$ by reverse inclusion, we see that it is the boolean algebra on the complement of $L \cup R$. Moreover, Lemma \ref{lem:B} shows that this partial order is commensurate with the partial order on $NC_B(n)$. Hence, we have a partition of $(NC_B(n),\leq)$ into boolean pieces. 

\begin{prp}\label{prop:LR-sbd}
We have the following partition of $NC_B(n)$:
\[ NC_B(n) = \bigcup_{(L,R) \in P_{B,\emptyset}} \eta( \B(L,R)).\]
In particular, $NC_{B}$ has a symmetric boolean decomposition.
\end{prp}

Figure~\ref{fig:B3-sbd} depicts the symmetric boolean decomposition corresponding to the description given in Proposition~\ref{prop:LR-sbd}. This partition of $NC_B(n)$ is also implicit in \cite[Section 3.5]{NP11}. It would interesting to provide a similar decomposition for noncrossing partitions of type D. See \cite[Section 6]{PetersenShards} for more comments on the matter.

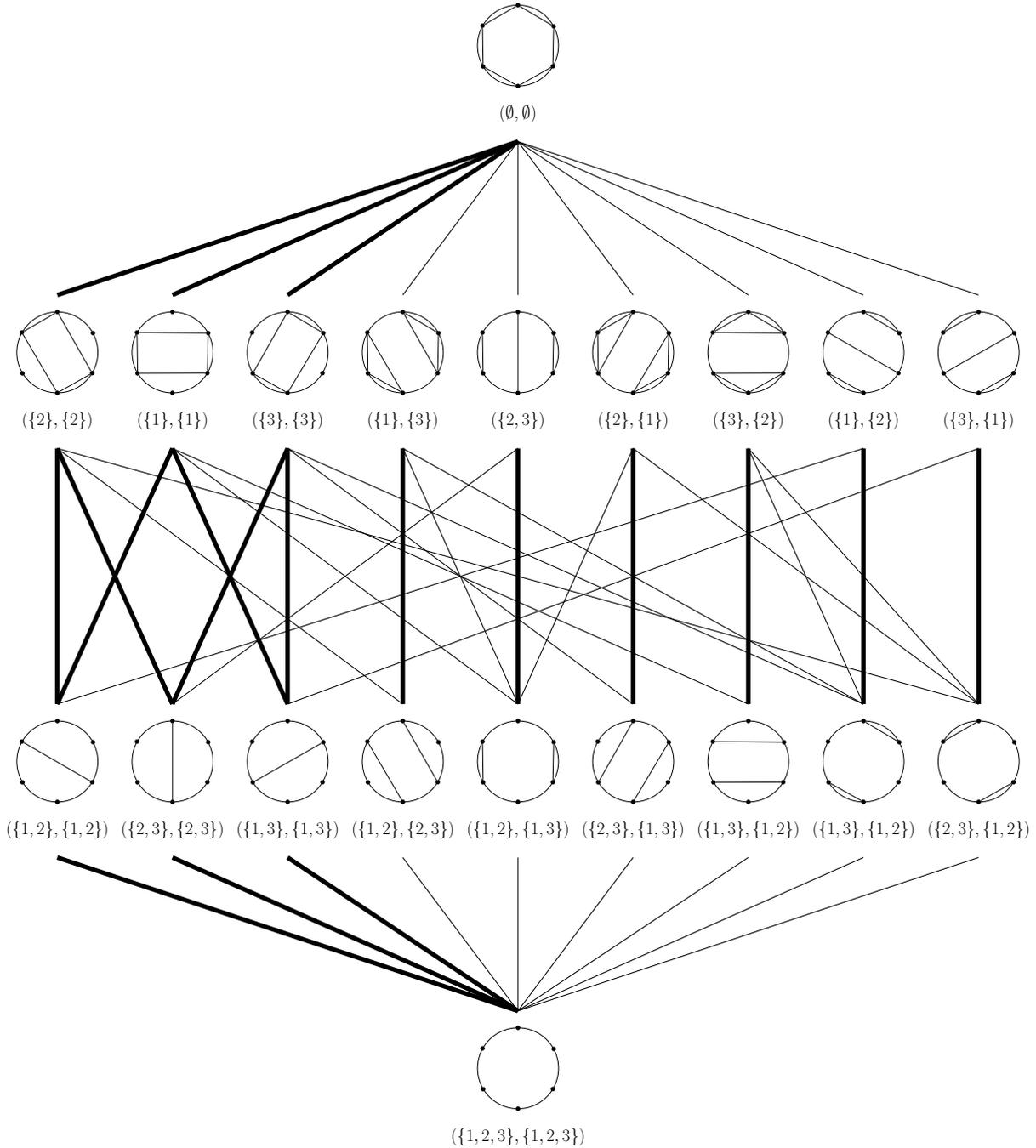
\begin{figure}[h!]
\begin{tikzpicture}[>=stealth',bend angle=45,auto, scale=.8]
\tikzstyle{state1}=[rectangle,draw=white,scale=.45]
\tikzstyle{state3}=[circle, fill=black, scale=.4]
\draw (9,-2) node[state1] (a0)
{ \begin{tikzpicture}[scale=.7]
\draw(9,3) circle (2.01cm);
\fill (9,5.01) node[state3] (1) {};
\fill (10.77,3.96) node[state3] (2) {};
\fill (10.73,1.97) node[state3] (3) {};
\fill (9,0.99) node[state3] (2) (-1){};
\fill (7.27,1.96) node[state3] (-2) {};
\fill (7.24,3.99) node[state3] (-3) {};
\draw(9,-.4) node {\Large $(\{1,2,3\},\{1,2,3\})$};
\end{tikzpicture}};
\draw (0,4) node[state1] (a1) 
{ \begin{tikzpicture}[scale=.7]
\draw(9,3) circle (2.01cm);
\fill (9,5.01) node[state3] (1) {};
\fill (10.77,3.96) node[state3] (2) {};
\fill (10.73,1.97) node[state3] (3) {};
\fill (9,0.99) node[state3] (2) (-1){};
\fill (7.27,1.96) node[state3] (-2) {};
\fill (7.24,3.99) node[state3] (-3) {};
\draw(9,-.4) node {\Large $(\{1,2\},\{1,2\})$};
\draw (-3)--(3);
\end{tikzpicture}};
\draw (2.25,4) node[state1] (a2) 
{ \begin{tikzpicture}[scale=.7]
\draw(9,3) circle (2.01cm);
\fill (9,5.01) node[state3] (1) {};
\fill (10.77,3.96) node[state3] (2) {};
\fill (10.73,1.97) node[state3] (3) {};
\fill (9,0.99) node[state3] (2) (-1){};
\fill (7.27,1.96) node[state3] (-2) {};
\fill (7.24,3.99) node[state3] (-3) {};
\draw(9,-.4) node {\Large $(\{2,3\},\{2,3\})$};
\draw (1)--(-1);
\end{tikzpicture}};
\draw (4.5,4) node[state1] (a3) 
{ \begin{tikzpicture}[scale=.7]
\draw(9,3) circle (2.01cm);
\fill (9,5.01) node[state3] (1) {};
\fill (10.77,3.96) node[state3] (2) {};
\fill (10.73,1.97) node[state3] (3) {};
\fill (9,0.99) node[state3] (2) (-1){};
\fill (7.27,1.96) node[state3] (-2) {};
\fill (7.24,3.99) node[state3] (-3) {};
\draw(9,-.4) node {\Large $(\{1,3\},\{1,3\})$};
\draw (2)--(-2);
\end{tikzpicture}};
\draw (6.75,4) node[state1] (a4) 
{ \begin{tikzpicture}[scale=.7]
\draw(9,3) circle (2.01cm);
\fill (9,5.01) node[state3] (1) {};
\fill (10.77,3.96) node[state3] (2) {};
\fill (10.73,1.97) node[state3] (3) {};
\fill (9,0.99) node[state3] (2) (-1){};
\fill (7.27,1.96) node[state3] (-2) {};
\fill (7.24,3.99) node[state3] (-3) {};
\draw(9,-.4) node {\Large $(\{1,2\},\{2,3\})$};
\draw(1)--(3);
\draw(-1)--(-3);
\end{tikzpicture}};
\draw (9,4) node[state1] (a5) 
{ \begin{tikzpicture}[scale=.7]
\draw(9,3) circle (2.01cm);
\fill (9,5.01) node[state3] (1) {};
\fill (10.77,3.96) node[state3] (2) {};
\fill (10.73,1.97) node[state3] (3) {};
\fill (9,0.99) node[state3] (2) (-1){};
\fill (7.27,1.96) node[state3] (-2) {};
\fill (7.24,3.99) node[state3] (-3) {};
\draw(9,-.4) node {\Large $(\{1,2\},\{1,3\})$};
\draw(2)--(3);
\draw(-2)--(-3);
\end{tikzpicture}};
\draw (11.25,4) node[state1] (a6) 
{ \begin{tikzpicture}[scale=.7]
\draw(9,3) circle (2.01cm);
\fill (9,5.01) node[state3] (1) {};
\fill (10.77,3.96) node[state3] (2) {};
\fill (10.73,1.97) node[state3] (3) {};
\fill (9,0.99) node[state3] (2) (-1){};
\fill (7.27,1.96) node[state3] (-2) {};
\fill (7.24,3.99) node[state3] (-3) {};
\draw(9,-.4) node {\Large $(\{2,3\},\{1,3\})$};
\draw(1)--(-2);
\draw(2)--(-1);
\end{tikzpicture}};
\draw (13.5,4) node[state1] (a7) 
{ \begin{tikzpicture}[scale=.7]
\draw(9,3) circle (2.01cm);
\fill (9,5.01) node[state3] (1) {};
\fill (10.77,3.96) node[state3] (2) {};
\fill (10.73,1.97) node[state3] (3) {};
\fill (9,0.99) node[state3] (2) (-1){};
\fill (7.27,1.96) node[state3] (-2) {};
\fill (7.24,3.99) node[state3] (-3) {};
\draw(9,-.4) node {\Large $(\{1,3\},\{1,2\})$};
\draw(2)--(-3);
\draw(3)--(-2);
\end{tikzpicture}};
\draw (15.75,4) node[state1] (a8) 
{ \begin{tikzpicture}[scale=.7]
\draw(9,3) circle (2.01cm);
\fill (9,5.01) node[state3] (1) {};
\fill (10.77,3.96) node[state3] (2) {};
\fill (10.73,1.97) node[state3] (3) {};
\fill (9,0.99) node[state3] (2) (-1){};
\fill (7.27,1.96) node[state3] (-2) {};
\fill (7.24,3.99) node[state3] (-3) {};
\draw(9,-.4) node {\Large $(\{1,3\},\{1,2\})$};
\draw(1)--(2);
\draw(-1)--(-2);
\end{tikzpicture}};
\draw (18,4) node[state1] (a9) 
{ \begin{tikzpicture}[scale=.7]
\draw(9,3) circle (2.01cm);
\fill (9,5.01) node[state3] (1) {};
\fill (10.77,3.96) node[state3] (2) {};
\fill (10.73,1.97) node[state3] (3) {};
\fill (9,0.99) node[state3] (2) (-1){};
\fill (7.27,1.96) node[state3] (-2) {};
\fill (7.24,3.99) node[state3] (-3) {};
\draw(9,-.4) node {\Large $(\{2,3\},\{1,2\})$};
\draw(1)--(-3);
\draw(3)--(-1);
\end{tikzpicture}};
\draw (0,12) node[state1] (b1) 
{ \begin{tikzpicture}[scale=.7]
\draw(9,3) circle (2.01cm);
\fill (9,5.01) node[state3] (1) {};
\fill (10.77,3.96) node[state3] (2) {};
\fill (10.73,1.97) node[state3] (3) {};
\fill (9,0.99) node[state3] (2) (-1){};
\fill (7.27,1.96) node[state3] (-2) {};
\fill (7.24,3.99) node[state3] (-3) {};
\draw(9,-.4) node {\Large $(\{2\},\{2\})$};
\draw (1)--(3);
\draw (3)--(-1);
\draw (-1)--(-3);
\draw (-3)--(1);
\end{tikzpicture}};
\draw (2.25,12) node[state1] (b2) 
{ \begin{tikzpicture}[scale=.7]
\draw(9,3) circle (2.01cm);
\fill (9,5.01) node[state3] (1) {};
\fill (10.77,3.96) node[state3] (2) {};
\fill (10.73,1.97) node[state3] (3) {};
\fill (9,0.99) node[state3] (2) (-1){};
\fill (7.27,1.96) node[state3] (-2) {};
\fill (7.24,3.99) node[state3] (-3) {};
\draw(9,-.4) node {\Large $(\{1\},\{1\})$};
\draw (-3)--(2);
\draw (-2)--(3);
\draw (2)--(3);
\draw (-2)--(-3);
\end{tikzpicture}};
\draw (4.5,12) node[state1] (b3) 
{ \begin{tikzpicture}[scale=.7]
\draw(9,3) circle (2.01cm);
\fill (9,5.01) node[state3] (1) {};
\fill (10.77,3.96) node[state3] (2) {};
\fill (10.73,1.97) node[state3] (3) {};
\fill (9,0.99) node[state3] (2) (-1){};
\fill (7.27,1.96) node[state3] (-2) {};
\fill (7.24,3.99) node[state3] (-3) {};
\draw(9,-.4) node {\Large $(\{3\},\{3\})$};
\draw (1)--(2);
\draw (2)--(-1);
\draw (-1)--(-2);
\draw (-2)--(1);
\end{tikzpicture}};
\draw (6.75,12) node[state1] (b4) 
{ \begin{tikzpicture}[scale=.7]
\draw(9,3) circle (2.01cm);
\fill (9,5.01) node[state3] (1) {};
\fill (10.77,3.96) node[state3] (2) {};
\fill (10.73,1.97) node[state3] (3) {};
\fill (9,0.99) node[state3] (2) (-1){};
\fill (7.27,1.96) node[state3] (-2) {};
\fill (7.24,3.99) node[state3] (-3) {};
\draw(9,-.4) node {\Large $(\{1\},\{3\})$};
\draw (1)--(2);
\draw (2)--(3);
\draw (1)--(3);
\draw (-1)--(-3);
\draw (-1)--(-2);
\draw (-2)--(-3);
\end{tikzpicture}};
\draw (9,12) node[state1] (b5) 
{ \begin{tikzpicture}[scale=.7]
\draw(9,3) circle (2.01cm);
\fill (9,5.01) node[state3] (1) {};
\fill (10.77,3.96) node[state3] (2) {};
\fill (10.73,1.97) node[state3] (3) {};
\fill (9,0.99) node[state3] (2) (-1){};
\fill (7.27,1.96) node[state3] (-2) {};
\fill (7.24,3.99) node[state3] (-3) {};
\draw(9,-.4) node {\Large $(\{2,3\})$};
\draw (2)--(3);
\draw (1)--(-1);
\draw (-3)--(-2);
\end{tikzpicture}};
\draw (11.25,12) node[state1] (b6) 
{ \begin{tikzpicture}[scale=.7]
\draw(9,3) circle (2.01cm);
\fill (9,5.01) node[state3] (1) {};
\fill (10.77,3.96) node[state3] (2) {};
\fill (10.73,1.97) node[state3] (3) {};
\fill (9,0.99) node[state3] (2) (-1){};
\fill (7.27,1.96) node[state3] (-2) {};
\fill (7.24,3.99) node[state3] (-3) {};
\draw(9,-.4) node {\Large $(\{2\},\{1\})$};
\draw (2)--(3);
\draw (3)--(-1);
\draw (2)--(-1);
\draw (-2)--(-3);
\draw (-3)--(1);
\draw (1)--(-2);
\end{tikzpicture}};
\draw (13.5,12) node[state1] (b7) 
{ \begin{tikzpicture}[scale=.7]
\draw(9,3) circle (2.01cm);
\fill (9,5.01) node[state3] (1) {};
\fill (10.77,3.96) node[state3] (2) {};
\fill (10.73,1.97) node[state3] (3) {};
\fill (9,0.99) node[state3] (2) (-1){};
\fill (7.27,1.96) node[state3] (-2) {};
\fill (7.24,3.99) node[state3] (-3) {};
\draw(9,-.4) node {\Large $(\{3\},\{2\})$};
\draw (1)--(2);
\draw (2)--(-3);
\draw (-3)--(1);
\draw (3)--(-2);
\draw (3)--(-1);
\draw (-1)--(-2);
\end{tikzpicture}};
\draw (15.75,12) node[state1] (b8) 
{ \begin{tikzpicture}[scale=.7]
\draw(9,3) circle (2.01cm);
\fill (9,5.01) node[state3] (1) {};
\fill (10.77,3.96) node[state3] (2) {};
\fill (10.73,1.97) node[state3] (3) {};
\fill (9,0.99) node[state3] (2) (-1){};
\fill (7.27,1.96) node[state3] (-2) {};
\fill (7.24,3.99) node[state3] (-3) {};
\draw(9,-.4) node {\Large $(\{1\},\{2\})$};
\draw (1)--(2);
\draw (3)--(-3);
\draw (-2)--(-1);
\end{tikzpicture}};
\draw (18,12) node[state1] (b9) 
{ \begin{tikzpicture}[scale=.7]
\draw(9,3) circle (2.01cm);
\fill (9,5.01) node[state3] (1) {};
\fill (10.77,3.96) node[state3] (2) {};
\fill (10.73,1.97) node[state3] (3) {};
\fill (9,0.99) node[state3] (2) (-1){};
\fill (7.27,1.96) node[state3] (-2) {};
\fill (7.24,3.99) node[state3] (-3) {};
\draw(9,-.4) node {\Large $(\{3\},\{1\})$};
\draw (1)--(-3);
\draw (2)--(-2);
\draw (3)--(-1);
\end{tikzpicture}};
\draw (9,18) node[state1] (c0) 
{ \begin{tikzpicture}[scale=.7]
\draw(9,3) circle (2.01cm);
\fill (9,5.01) node[state3] (1) {};
\fill (10.77,3.96) node[state3] (2) {};
\fill (10.73,1.97) node[state3] (3) {};
\fill (9,0.99) node[state3] (2) (-1){};
\fill (7.27,1.96) node[state3] (-2) {};
\fill (7.24,3.99) node[state3] (-3) {};
\draw (1)--(2);
\draw (2)--(3);
\draw (3)--(-1);
\draw (-1)--(-2);
\draw (-2)--(-3);
\draw (-3)--(1);
\draw(9,-.4) node {\Large $(\emptyset,\emptyset)$};
\end{tikzpicture}};
\draw [line width=2pt] (0,5.5)--(0,10.5); 
\draw [line width=2pt] (0,5.5)--(2.25,10.5); 
\draw (0,5.5)--(15.75,10.5); 
\draw [line width=2pt] (2.25,5.5)--(0,10.5); 
\draw [line width=2pt] (2.25,5.5)--(4.5,10.5); 
\draw (2.25,5.5)--(9,10.5); 
\draw [line width=2pt] (4.5,5.5)--(2.25,10.5); 
\draw [line width=2pt] (4.5,5.5)--(4.5,10.5); 
\draw (4.5,5.5)--(18,10.5); 
\draw [line width=2pt] (6.75,5.5)--(6.75,10.5); 
\draw (6.75,5.5)--(0,10.5); 
\draw [line width=2pt] (9,5.5)--(9,10.5); 
\draw (9,5.5)--(2.25,10.5); 
\draw (9,5.5)--(6.75,10.5); 
\draw (9,5.5)--(11.25,10.5); 
\draw [line width=2pt] (11.25,5.5)--(11.25,10.5); 
\draw (11.25,5.5)--(4.5,10.5); 
\draw [line width=2pt] (13.5,5.5)--(13.5,10.5); 
\draw (13.5,5.5)--(2.25,10.5); 
\draw [line width=2pt] (15.75,5.5)--(15.75,10.5); 
\draw (15.75,5.5)--(4.5,10.5); 
\draw (15.75,5.5)--(6.75,10.5); 
\draw (15.75,5.5)--(13.5,10.5); 
\draw [line width=2pt] (18,5.5)--(18,10.5); 
\draw (18,5.5)--(0,10.5); 
\draw (18,5.5)--(11.25,10.5); 
\draw (18,5.5)--(13.5,10.5); 
\draw [line width=2pt] (9,-.5)--(0,2.5);
\draw [line width=2pt] (9,-.5)--(2.25,2.5);
\draw [line width=2pt] (9,-.5)--(4.5,2.5);
\draw (9,-.5)--(6.75,2.5);
\draw (9,-.5)--(9,2.5);
\draw (9,-.5)--(11.25,2.5);
\draw (9,-.5)--(13.5,2.5);
\draw (9,-.5)--(15.75,2.5);
\draw (9,-.5)--(18,2.5);
\draw [line width=2pt] (9,16.5)--(0,13.5);
\draw [line width=2pt] (9,16.5)--(2.25,13.5);
\draw [line width=2pt] (9,16.5)--(4.5,13.5);
\draw (9,16.5)--(6.75,13.5);
\draw (9,16.5)--(9,13.5);
\draw (9,16.5)--(11.25,13.5);
\draw (9,16.5)--(13.5,13.5);
\draw (9,16.5)--(15.75,13.5);
\draw (9,16.5)--(18,13.5);
\end{tikzpicture}
\caption{The lattice $NC_{B}(n)$. Boldface lines indicate the symmetric boolean decomposition exhibited in Proposition~\ref{prop:LR-sbd}.}
\label{fig:B3-sbd}
\end{figure}

Two comments are in order: 
\begin{enumerate}
\item[(i)] Since $|NC_{B}(n)|=\binom{2n}{n}$ (see, e.g.,~\cite{R97}), Proposition~\ref{prop:LR-sbd} yields
\[
\binom{2n}{n}=\sum_{i=0}^{\lfloor n/2\rfloor}\binom{n}{i}\binom{n-i}{i}2^{n-2i}.
\]
\item[(ii)] Moreover, recalling the correspondence between type B noncrossing partitions and elements in the interval $[e,c]_B$ in the absolute order on the hyperoctahedral group $B_{n}$, we have \[ \sum_{\pi \in NC_B(n)} t^{\rk(\pi)} = \sum_{\sigma \in [e,c]_B} t^{\ell'(\sigma)} = \sum_{ (L,R) \in P_{B,\emptyset}} t^{|L|}(1+t)^{n-2|L|} = \sum_{0\leq 2i \leq n} \binom{n}{i,i}t^i(1+t)^{n-2i}.\] This expression was also noted in \cite{PRW08}.

It would then seem natural to see if we can study the joint distribution of length and reflection length via this decomposition, but computer experimentation shows that the polynomial \[ \sum_{\sigma \in [e,c]_B} q^{\ell(\sigma)}t^{\ell'(\sigma)} \] does not have an expansion akin to the one for type A given in Theorem \ref{thm:main}.
\end{enumerate}

\bigskip

\noindent\textbf{Acknowledgements.}
The authors would like to thank Drew Armstrong for early discussions on the topic, as well as the organizers of the Fall 2010 Banff workshop on quasisymmetric functions, where early discussions on the topic took place. We also thank Christian Stump and Bridget Tenner for valuable comments on earlier versions of the paper.

\end{document}